\documentclass[10pt, oneside]{article}

\usepackage{mathrsfs}
\usepackage{amsmath,amsfonts,amssymb,amsthm}
\usepackage{caption}
\usepackage{subfigure}
\usepackage{cite}
\usepackage{color}
\usepackage{graphicx}
\usepackage{subfigure}
\usepackage{float}
\usepackage{paralist}
\usepackage{indentfirst}
\usepackage{authblk}
\usepackage[colorlinks,linkcolor=red,citecolor=blue]{hyperref}

\definecolor{red}{rgb}{1.00,0.00,0.00}
\definecolor{blue}{rgb}{0.00,0.00,0.63}
\definecolor{black}{rgb}{0.00,0.00,0.00}
\definecolor{purple}{rgb}{0.00,1.00,0.00}
\definecolor{pink}{rgb}{0.95,0.01,0.08}

\usepackage{color}
\usepackage[T1]{fontenc}

\newtheorem{theorem}{Theorem}[section]
\newtheorem{lemma}{Lemma}[section]
\newtheorem{proposition}{Proposition}[section]

\newtheorem{remark}{Remark}[section]
\newtheorem{definition}{Definition}[section]

\numberwithin{equation}{section}

\begin{document}

\title{{\LARGE \textbf{Exponential Stability of the Inhomogeneous Navier-Stokes-Vlasov System in Vacuum}}}

\author[a,b]{Hai-Liang Li}

\author[c]{Ling-Yun Shou}

\author[a]{Yue Zhang\thanks{Corresponding author}}

\affil[a]{School of Mathematical Sciences, Capital Normal University, Beijing  100048, P. R. China}

\affil[b]{Academy for Multidisciplinary Studies, Capital Normal University, Beijing  100048, P. R. China}

\affil[c]{School of Mathematics, Nanjing University of Aeronautics and
Astronautics, Nanjing  211106, P. R. China}

\date{}
\renewcommand*{\Affilfont}{\small\it}
\maketitle

\begin{abstract}
In this paper, we study the asymptotic behaviors of solutions to the inhomogeneous Navier-Stokes-Vlasov system in $\mathbb{R}^{3}\times\mathbb{R}^{3}$, where the initial fluid density is allowed to vanish. We establish the uniform bound of the macroscopic density associated with the distribution function and prove the global existence and uniqueness of strong solutions to the Cauchy problem with vacuum for either small initial energy or large viscosity coefficient. The uniform boundedness and the presence of vacuum enable us to show that as the time evolves, the fluid velocity decays, while the distribution function concentrates towards a Dirac measure in velocity centred at $0$, with an exponential rate. In order to overcome the degeneracy in the momentum equations, we develop an energy argument based on higher order functional inequalities designed for fluid-particle coupled structures.
\end{abstract}

\footnotetext{E-mails: hailiang.li.math@gmail.com (H.-L. Li), shoulingyun11@gmail.com (L.-Y. Shou), yuezhangmath@126.com (Y. Zhang).}
\footnotetext{Funding: The first and third authors are partially supported by the National Natural Science Foundation of China (11931010, 12226326), the Beijing Scholar Foundation of Beijing Municipal Committee, and the key research project of Academy for Multidisciplinary Studies, Capital Normal University. The second author is supported by the National Natural Science Foundation of China (12301275) and the China Postdoctoral Science Foundation (2023M741694).}

\section{Introduction}

The transport of small particles dispersed in dense fluids can be described by the fluid-particle models which have a wide range of applications including chemical engineering, pollution settling processes, medical treatments, diesel engines, etc \cite{Amsden1989,Berres2003,Caflisch1983,Jabin2000,ORourke,Williams}.

We investigate the inhomogeneous Navier-Stokes-Vlasov system for the fluid-particle motion:
\begin{equation}\label{NSV}
\left\{
\begin{aligned}
&\rho_{t}+{\rm{div}}_{x}(\rho u)=0,\\
&(\rho u)_{t}+{\rm{div}}_{x}(\rho u\otimes u)+\nabla_{x}P=\mu\Delta_{x}u-\int_{\mathbb{R}^{3}}\kappa(u-v)f{\rm{d}}v,\\
&{\rm{div}}_{x}u=0,\\
&f_{t}+v\cdot\nabla_{x}f+{\rm{div}}_{v}(\kappa (u-v)f)=0,
\end{aligned}
\right.
\end{equation}
where $\rho=\rho(t,x)\geq 0$, $u=(u_{1},u_{2},u_{3})(t,x)$, and $P=P(t,x)$ denote the density, the velocity, and the pressure of the fluid, respectively, at the time $t\in\mathbb{R}_{+}$ and the position $x=(x_{1},x_{2},x_{2})\in\mathbb{R}^{3}$, and $f=f(t,x,v)\geq 0$ is the distribution function of particles moving with the velocity $v=(v_{1},v_{2},v_{3})\in\mathbb{R}^{3}$. In addition, the constants $\mu>0$ and $\kappa>0$ stand for the viscosity coefficient and the drag force coefficient, respectively.
In this paper, we consider the Cauchy problem of \eqref{NSV} subject to the initial data
\begin{equation}\label{inda}
\rho(0,x)=\rho_{0}(x),~~u(0,x)=u_{0}(x),~~f(0,x,v)=f_{0}(x,v),\quad(x,v)\in\mathbb{R}^{3}\times\mathbb{R}^{3}.
\end{equation}

The Navier-Stokes-Vlasov systems have been intensively investigated with much significant mathematical progress, refer to \cite{Anoshchenko,Boudin2009,Hofer,Carrillo2011,Chae2011,Yu,Boudin2017,HanKwan2020,HanK2020,Goudon2010,Wang2015,Carrillo2016,Su2023,Ertzbischoff} and the references therein. The global existence of weak solutions for the unsteady Stokes-Vlasov system was proved by Hamdache \cite{Hamdache} in bounded domains with specular reflection boundary conditions. The convergence of the inertialess limit for the Vlasov equation coupled with the steady Stokes flow has been rigorously justified by H\"ofer \cite{Hofer}. For the incompressible Navier-Stokes-Vlasov model, we mention the known results \cite{Anoshchenko,Boudin2009,Carrillo2016,Yu,Boudin2017} about global existence of weak solutions in the case of bounded domains, periodic domains, or time-dependent domains. Han-Kwan et al. \cite{HanK2020} proved the uniqueness of weak solutions with  a mild decay assumption on the initial distribution function. Glass et al. \cite{Glass2018} investigated the asymptotic stability of stationary states in a pipe with partially absorbing boundary conditions. Under the condition that the initial energy and the $\dot{H}^{\frac{1}{2}}$-norm of the initial velocity are sufficiently small, Han-Kwan et al. \cite{HanKwan2020} established the uniform-in-time boundedness of the macroscopic density for the particles and proved the exponential-in-time stability of global weak solutions to the initial value problem in spatial periodic domains. The methods in \cite{HanKwan2020} were then applied to the study of the large-time behaviors of weak solutions in spatial bounded domains with exponential decay rates \cite{Ertzbischoff2021} and in the whole space/half space with algebraic decay rates \cite{HanK,Ertzbischoff}. Very recently, Han-Kwan and Michel \cite{HanK2021} developed this framework to investigate hydrodynamic limits of the incompressible Navier-Stokes-Vlasov system in high friction regimes.

For the inhomogeneous Navier-Stokes-Vlasov system \eqref{NSV}, Wang and Yu \cite{Wang2015} showed the global existence of weak solutions to the initial boundary value problem, Choi and Kwon \cite{Choi2015} studied the large-time solvability of strong solutions to the initial value problem with small initial data in either spatial periodic domains or the whole space. The exponential time-decay rate in the spatial periodic space was also obtained in \cite{Choi2015}, provided that the solutions satisfy some uniform-in-time estimates. Su et al. \cite{Su2023} justified the hydrodynamic limit of system \eqref{NSV} to a macroscopical two-fluid system. Furthermore, much important progress has been made recently on the well-posedness and asymptotic behaviors of solutions for other related fluid-particle models, cf. \cite{Li2021,Li2023,Carrillo2006,Carrillo2011,Chae2011,Chae2013,Choi2021,Mellet2008,Mellet2007,Duan2013,Goudon2010,Li2017,Cao2021,Li2022} as well as references therein.

In the present paper, we aim to analyze the influence of vacuum on large-time asymptotics of global solutions for the inhomogeneous Navier-Stokes-Vlasov system \eqref{NSV} in $\mathbb{R}^{3}\times\mathbb{R}^{3}$. Based on the structure of the fluid-particle interactions and the appearance of the vacuum, we establish some uniform a-priori estimates with respect to time, which allow us not only to obtain the global-in-time existence of weak and strong solutions to the Cauchy problem \eqref{NSV}--\eqref{inda} with vacuum, but also to justify the {\emph{exponential}} decay-in-time property. The exponential stability for \eqref{NSV} observed in this paper is different from the algebraic decay rates for the homogeneous case in unbounded domains \cite{HanK,Ertzbischoff}.

\vspace{2mm}

Before stating the main results, we present some notations below.

\noindent\textbf{Notations.} 
For $p \in [1,+\infty]$ and integer $k \geq 0$, the simplified Lebesgue and Sobolev spaces are given by
\begin{equation*}
\left\{
\begin{aligned}
&L^{p}:=L^{p}(\mathbb{R}^{3}),~~ W^{k,p}:=W^{k,p}(\mathbb{R}^{3}),~~ H^{k}:= W^{k,2},\\
&D^{1}=\{u\in L^6|\,\nabla u\in L^{2}\},~~ C_{0,\sigma}^{\infty}:=\{u\in C_{0}^{\infty}(\mathbb{R}^{3})|\,{\rm{div}}\,u=0\},~~ D_{0,\sigma}^{1}:=\overline{C_{0,\infty}^{\infty}}^{\|\cdot\|_{D^{1}}},\\
&L_{x,v}^{p}:=L^{p}(\mathbb{R}^{3}\times \mathbb{R}^{3}) ,~~ W_{x,v}^{k,p}:=W^{k,p}(\mathbb{R}^{3}\times \mathbb{R}^{3}).
\end{aligned}
\right.
\end{equation*}
We define the kinetic energy and dissipation of system \eqref{NSV} as
\begin{equation}\label{enrg}
E(t):=\frac{1}{2}\int_{\mathbb{R}^{3}}\rho|u|^{2}{\rm{d}}x+\frac{1}{2}\int_{\mathbb{R}^{3}\times\mathbb{R}^{3}}|v|^{2}f{\rm{d}}v{\rm{d}}x,
\end{equation}
and
\begin{equation}\label{dspt}
D(t):=\int_{\mathbb{R}^{3}}\mu|\nabla u|^{2}{\rm{d}}x+\int_{\mathbb{R}^{3}\times\mathbb{R}^{3}}\kappa|u-v|^{2}f{\rm{d}}v{\rm{d}}x,
\end{equation}
respectively. For any strong solution $(\rho,u,f)$ to the system \eqref{NSV}, one has the basic energy inequality
\begin{equation}\label{enrgineq}
\frac{{\rm{d}}}{{\rm{d}}t}E(t)+D(t)\leq 0.
\end{equation}
In addition, we define the corresponding initial energy
\begin{equation}\label{inenrg}
E_{0}:=E(0)=\frac{1}{2}\int_{\mathbb{R}^{3}}\rho_{0}|u_{0}|^{2}{\rm{d}}x+\frac{1}{2}\int_{\mathbb{R}^{3}\times\mathbb{R}^{3}}|v|^{2}f_{0}{\rm{d}}v{\rm{d}}x,
\end{equation}
and a constant corresponding to the initial dissipation
\begin{equation}\label{indspt}
M:=\int_{\mathbb{R}^{3}}|\nabla u_{0}|^{2}{\rm{d}}x+\int_{\mathbb{R}^{3}\times\mathbb{R}^{3}}|u_{0}-v|^{2}f_{0} {\rm{d}}v{\rm{d}}x.
\end{equation}
Let $n_{f}$, $j_{f}$ and $e_{f}$ be the macroscopical density, momentum and energy, respectively, related to the
moments of the distribution function $f$ in \eqref{NSV} as follows
\begin{align*}
n_{f}(t,x):=\int_{\mathbb{R}^{3}}f(t,x,v) {\rm{d}}v,\quad j_{f}(t,x):=\int_{\mathbb{R}^{3}}vf(t,x,v) {\rm{d}}v,\quad e_{f}(t,x):=\frac{1}{2}\int_{\mathbb{R}^{3}}|v|^{2}f(t,x,v) {\rm{d}}v.
\end{align*}
Throughout this paper, $C>0$ ($C_{T}>0$) denotes some positive constants independent (dependent) of $E_{0}$, $\mu$, $\kappa$ and time.

\vspace{2mm}

First, we have the global existence, uniqueness, and large-time behaviors of strong solutions to the Cauchy problem \eqref{NSV}--\eqref{inda} with vacuum. We show that the fluid velocity $u$ and the distribution function $f$ asymptotically converge at an exponential rate to their equilibrium state $0$ and $n^{*}\otimes\delta_{v}$, respectively.

\begin{theorem}\label{main1}
Let $\mu>0$, $\kappa\geq1$ and $R_{0}>0$ be given constants. Assume that the initial data $(\rho_{0},u_{0},f_{0})$ satisfies
\begin{equation}\label{indacd}
\left\{
\begin{aligned}
&0\leq \rho_{0}\in L^{\frac{3}{2}}\cap L^{\infty}\cap H^{1},\\
& \sqrt{\rho_{0}}u_{0}\in L^{2},~~ u_{0}\in D^{1}_{0,\sigma},\\
&0\leq f_{0}\in  L^{1}_{x,v}\cap L^{\infty}_{x,v} \cap W^{1,\frac{3}{2}}_{x,v},\\
&{\rm{Supp}}_{v}f_{0}(x,v)\subset\{v\in\mathbb{R}^{3}|\,|v|\leq R_{0}\}.
\end{aligned}
\right.
\end{equation}
Then, there exists a constant $\varepsilon_{0}>0$ depending on $M$, $\|\rho_{0}\|_{L^{\frac{3}{2}}\cap L^{\infty}}$, $\|f_{0}\|_{L^{1}_{x,v}}$ and $\|(1+|v|^{2})f_{0}\|_{ L^{1}_{v}(L^{\infty}_{x})}$ such that if
\begin{equation}\label{inenrgcd}
E_{0}\leq\min\{\mu^\frac{10}{7},\,\mu^{13}\}\kappa^{-16}\varepsilon_{0},
\end{equation}
the Cauchy problem \eqref{NSV}--\eqref{inda} admits a unique global solution $(\rho,u,f)$ satisfying for any $T>0$ that
\begin{equation}\label{result1}
\left\{
\begin{aligned}
&\rho\in C([0,T];L^{\frac{3}{2}}\cap H^{1})\cap L^{\infty}(0,T;L^{\infty}), ~~ \sqrt{\rho}u_{t}\in L^{2}(0,T;L^{2}),\\
&u\in L^{\infty}(0,T;D^{1}_{0,\sigma})\cap L^{1}(0,T;W^{1,\infty}),~~ P\in L^{2}(0,T;H^{1}),\\
&\sqrt{t}(\nabla u, P)\in L^{\infty}(0,T;H^{1})\cap L^{2}(0,T;W^{1,6}), \\
&\sqrt{t} \sqrt{\rho}u_{t}\in L^{\infty}(0,T;L^{2})\cap L^{2}(0,T;L^{6}),\\
&f\in C([0,T];L_{x,v}^{1}\cap W^{1,\frac{3}{2}}_{x,v})\cap L^{\infty}(0,T;L^{\infty}_{x,v}),\\
\end{aligned}
\right.
\end{equation}
and for any $(t,x,v)\in(0,T]\times\mathbb{R}^{3}\times\mathbb{R}^{3}$ that
\begin{equation}\label{result2}
\left\{
\begin{aligned}
&0\leq \rho(t,x)\leq \|\rho_{0}\|_{L^{\infty}},~~ \|\nabla \rho(t)\|_{L^{2}}\leq 2\|\nabla \rho_{0}\|_{L^{2}},\\
&0\leq f(t,x,v)\leq e^{3\kappa T}\|f_{0}\|_{L^{\infty}_{x,v}},~~ 0\leq n_{f}(t,x)\leq 2\|f_{0}\|_{L^{1}_{v}(L^{\infty}_{x})}.
\end{aligned}
\right.
\end{equation}
Moreover, for any $t\in[1,+\infty)$, there exists an asymptotic profile $n^{*}$ of $n_{f}$ such that
\begin{equation}\label{result3}
\left\{
\begin{aligned}
&E(t)+D(t)+\|\nabla^{2} u(t)\|_{L^{2}}+\|P(t)\|_{H^{1}}\leq C_{1}e^{-\lambda_{1} t},\\
&W_1(f(t),n^{*}\otimes\delta_{v})+W_1(n_{f}(t),n^{*})\leq C_{1}e^{-\lambda_{1} t},\\
&\|j_{f}(t)\|_{L^{\infty}}+\|e_{f}(t)\|_{L^{\infty}}\leq C_{1}e^{-\lambda_{1} t},\\
&{\rm{Supp}}_{v}f(t,x,v)\subset\{v\in\mathbb{R}^{3}~|~|v|\leq (R_{0}+C_{1})e^{-\lambda_{1}t}\},
\end{aligned}
\right.
\end{equation}
where $C_{1}>0$ and $\lambda_{1}>0$ are two constants dependent of $\mu$, $\kappa$ but independent of the time $t$, $W_1$ represents the $1$-Wasserstein distance {\rm(}refer to Definition {\rm\ref{defwass}}{\rm)}, and $\delta_{v}$ denotes the Dirac measure in velocity supported at $\{v=0\}$.
\end{theorem}

\begin{remark}
Some comments are in order.
\begin{itemize}
\item In \eqref{result3}, the large-time behavior of the energy $E(t)$ provides the $L^{2}$-decay of $\sqrt{\rho}u$ but without the $L^{2}$-decay of $u$ due to the influence of vacuum. However, we are able to derive the $L^{6}\cap L^{\infty}$-decay of $u$ since we capture the exponential rates of the dissipation $D(t)$ and the high-order norm $\|\nabla^{2}u(t)\|_{L^{2}}$ for large times. Indeed, it follows from the Gagliardo-Nirenberg-Sobolev inequalities that
\begin{equation*}
\|u(t)\|_{L^{6}\cap L^{\infty}}\leq C_{1}e^{-\lambda_{1}t}.
\end{equation*}
From the definition of $D(t)$, we also obtain
\begin{equation*}
 \|\sqrt{f}(u-v)(t)\|_{L^{2}_{x,v}}^{2}\leq C_{1}e^{-\lambda_{1}t},
\end{equation*}
which implies the alignment between the fluid velocity $u$ and the particle velocity $v$ as the time evolves.

\item Different from the interesting works {\rm\cite{Ertzbischoff,HanK}} concerning algebraic decay rates of the homogeneous incompressible Navier-Stokes-Vlasov equations {\rm(}i.e., $\rho\equiv 1$ or $\rho$ is bounded from below in \eqref{NSV}{\rm)}, we obtain an exponential decay rate in the unbounded spatial domain $\mathbb{R}^{3}$. This comes from the fact that the fluid energy $\frac{1}{2}\|\sqrt{\rho}u \|_{L^{2}}^{2}$ can be completely controlled by the viscosity dissipation $\mu\|\nabla u \|_{L^{2}}^{2}$ due to the $L^{\frac{3}{2}}$-integrability of $\rho$ in vacuum {\rm(}refer to Proposition {\rm\ref{prop1}}{\rm)}. The exponential decay rate also allows us to establish the uniform large-time estimates without assuming additional $L^{1}$ regularity on the initial data as in {\rm\cite{HanK}}. Furthermore, our  result presents a new comparison to the spatial periodic case {\rm\cite{Choi2015}}.

\item For any given $\kappa\geq 1$, the mild condition \eqref{inenrgcd} is satisfied, provided that either the initial energy $E_{0}$ is suitable small or the viscosity coefficient $\mu$ is large enough.

\item For the global existence and uniqueness in Theorem {\rm\ref{main1}}, the compact support condition $\eqref{indacd}_{4}$ can be relaxed to $(1+|v|^{q})f_{0}\in L^{\infty}_{x,v}$ for any $q>5$.

\item It is interesting to extend the linear drag force term $F^{0}_{d}=\kappa (u-v)$ to the general form
\begin{equation*}
F_{d}^{\alpha}=\kappa \rho^{\alpha}(u-v)~~ \text{with}~~ \alpha>0.
\end{equation*}
The main difficulty is that, in order to ensure the validity of the change of velocity variable for $n_{f}$ as in Lemma {\rm\ref{lem2}}, one must require the lower bound and Lipschitz regularity of the fluid density $\rho$. This is a work in progress.
\end{itemize}
\end{remark}

\newpage

We further investigate the global existence and large-time behaviors of weak solutions to the Cauchy problem \eqref{NSV}--\eqref{inda} with only bounded fluid density and distribution function.

\begin{definition}\label{defwkslt}
For any time $T>0$, $(\rho,u,f)$ is said to be a global weak solution to the Cauchy problem \eqref{NSV}--\eqref{inda} if the following properties hold

\begin{itemize}
\item [{\rm (1)}] $(\rho,u,f)$ satisfies
\begin{equation*}
\left\{
\begin{aligned}
&0\leq \rho \in C([0,T];L^{\frac{3}{2}})\cap L^{\infty}(0,T;L^{\infty}),\\
&\rho u\in C([0,T];H^{-1})\cap L^{\infty}(0,T;L^{2}),~~ u\in L^{2}(0,T;D_{0,\sigma}^{1}),\\
&0\leq f\in C([0,T];L^{1}_{x,v})\cap L^{\infty}(0,T;L^{\infty}_{x,v});\\
\end{aligned}
\right.
\end{equation*}

\item [{\rm (2)}] $(\rho,u,f)$ solves the Cauchy problem \eqref{NSV}--\eqref{inda} in the sense of distributions;

\item [{\rm (3)}] For a.e. $t\in (0,T]$, the energy inequality \eqref{enrgineq} holds.
\end{itemize}
\end{definition}

\begin{theorem}\label{main2}
Let $\mu>0$, $\kappa\geq1$ and $q>5$. Assume that the initial data $(\rho_{0},u_{0},f_{0})$ satisfies
\begin{equation}\label{indacd2}
\left\{
\begin{aligned}
&0\leq \rho_{0}\in L^{\frac{3}{2}}\cap L^{\infty},\\
&\sqrt{\rho_{0}}u_{0}\in L^{2},~~ u_{0}\in D^{1}_{0,\sigma},\\
&0\leq f_{0}\leq L^{1}_{x,v}\cap L^{\infty}_{x,v},~~ |v|^{q}f_{0}\in  L^{\infty}_{x,v}.
\end{aligned}
\right.
\end{equation}
Then, there exists a constant $\varepsilon_{0}^{*}>0$ depending on $M$, $\|\rho_{0}\|_{L^{\frac{3}{2}}\cap L^{\infty}}$, $\|f_{0}\|_{L^{1}_{x,v}}$ and $\|(1+|v|^{2})f_{0}\|_{ L^{1}_{v}(L^{\infty}_{x})}$ such that if
\begin{equation}\label{inenrgcd2}
E_{0}\leq\min\{\mu^\frac{10}{7},\,\mu^{13}\}\kappa^{-16}\varepsilon_{0}^{*},
\end{equation}
the Cauchy problem \eqref{NSV}--\eqref{inda} admits a global weak solution $(\rho,u,f)$ in the sense of Definition {\rm\ref{defwkslt}}. In addition, for any $T>0$, we have
\begin{equation}
\left\{
\begin{aligned}
&u\in L^{\infty}(0,T;D^{1}_{0,\sigma})\cap L^{1}(0,T;W^{1,\infty}),~~ P\in L^{2}(0,T;H^{1}),\\
&\sqrt{t}(\nabla u,P)\in L^{\infty}(0,T;H^{1})\cap L^{2}(0,T;W^{1,6}), \\
&\sqrt{t} \sqrt{\rho}u_{t}\in L^{\infty}(0,T;L^{2})\cap L^{2}(0,T;L^{6}).
\end{aligned}
\right.
\end{equation}
Furthermore, it holds for any $t\in[1,+\infty)$ that
\begin{equation}
\left\{
\begin{aligned}
&E(t)+D(t)+\|\nabla^{2}u(t)\|_{L^{2}}+\|P(t)\|_{H^{1}}\leq C_{2}e^{-\lambda_{2} t},\\
&W_{1}(f(t),n^{*}\otimes\delta_{v})+W_{1}(n_{f}(t),n^{*})\leq C_{2}e^{-\lambda_{2} t},\\
&\|j_{f}(t)\|_{L^{\infty}}+\|e_{f}(t)\|_{L^{\infty}}\leq C_{2}e^{-\lambda_{2} t},\label{r2}
\end{aligned}
\right.
\end{equation}
where $C_{2}>0$ and $\lambda_{2}>0$ are two constants dependent of $\mu$, $\kappa$ but independent of the time $t$, $W_{1}$ represents the $1$-Wasserstein distance {\rm(}refer to Definition {\rm\ref{defwass}}{\rm)}, $n^{*}$ is an asymptotic profile of $n_{f}$, and $\delta_{v}$ denotes the Dirac measure in velocity supported at $\{v=0\}$.
\end{theorem}

We would like to mention that there is a considerable amount of recent improvements on the well-posedness and large-time description of solutions to the inhomogeneous Navier-Stokes equations (i.e., the distribution function $f=0$ in \eqref{NSV}).  Ka\v{z}ihov \cite{Kazihov} showed that, in the absence of vacuum, at least a global weak solution exists in the energy space. The no vacuum assumption was later removed by Simon \cite{Simon}. Lions \cite{Lions} introduced the renormalized solutions and extended previous results to the case of density-dependent viscosity. For smooth initial data, the Cauchy problem or the initial boundary value problem admits a unique solution locally in time or globally in time when the initial velocity is small enough, cf. \cite{Ladyzenskaja,Cho2004,Paicu2013}. The reader also refers to \cite{Danchin,Abidi,Abidi2012,Abidi2007,Danchin2012,Zhang,Danchin2023} concerning the advances in the critical regularity settings. When the initial density is allowed to vanish, Craig et al. \cite{Craig2013} established the global well-posedness of the Cauchy problem in three dimensions for regular initial data with the suitable small $\dot{H}^{\frac{1}{2}}$-norm of $u_{0}$. Danchin and Mucha \cite{Danchin2019} addressed the global existence and uniqueness issue for the only bounded density and provided an answer to Lions' question in \cite{Lions} concerning the evolution of a drop in vacuum. In addition, He et al. \cite{He2021} proved an exponential decay result for classical solutions with vacuum in the three-dimensional whole space.

\vspace{2mm}

We explain the main strategies and ideas to prove Theorems \ref{main1} and \ref{main2} about the global existence and large-time behaviors of solutions to the Cauchy problem \eqref{NSV}--\eqref{inda}. To begin with, we prove the coercive estimate
\begin{equation}\label{invsenrg}
D(t)\geq\lambda E(t)~~ \text{with}~~ \lambda\sim \frac{\min\{\mu,\kappa\}}{1+\displaystyle\sup_{t>0}\|(\rho,n_{f})(t)\|_{L^\frac{3}{2}}}.
\end{equation}
By \eqref{invsenrg} and the energy inequality \eqref{enrgineq}, we obtain the exponential stability of the basic energy $E(t)$ (refer to Proposition \ref{prop1}).
Since the transport nature in $\eqref{NSV}_{1}$ with vacuum leads to the uniform evolution of the $L^{\frac{3}{2}}$-norm for the fluid density $\rho$, our key ingredient is to show the uniform bound
\begin{equation}\label{nfl1h}
\sup_{t>0}\|n_{f}(t)\|_{L^\frac{3}{2}}<\infty.
\end{equation}
To achieve it, we rewrite the Vlasov equation $\eqref{NSV}_{4}$ as
\begin{equation}\label{dstrbt}
f(t,x,v)=e^{3\kappa t}f_{0}(X(0;t,x,v),V(0;t,x,v)),
\end{equation}
where the characteristic curves $X(\tau;t,x,v)$ and $V(\tau;t,x,v)$ are solved by
\begin{equation}\label{chrcrv}
\left\{
\begin{aligned}
&\frac{{\rm{d}}}{{\rm{d}}\tau}X(\tau;t,x,v)=V(\tau;t,x,v),\\
&\frac{{\rm{d}}}{{\rm{d}}\tau}V(\tau;t,x,v)=\kappa(u(\tau,X(\tau;t,x,v))-V(\tau;t,x,v)),\\
&X(t;t,x,v)=x,~~ V(t;t,x,v)=v.
\end{aligned}
\right.
\end{equation}
Note that the formula \eqref{dstrbt} implies the time growth of the distribution function $f$, so one may not obtain \eqref{nfl1h} directly. Since the $L^{1}$-norm of $n_{f}$ is conserved, it suffices to obtain the uniform upper bound of $n_{f}$. As in  \cite{HanKwan2020,Ertzbischoff2021,HanK}, we need to derive  the  upper bound $2e^{-3\kappa t}$ of the Jacobian matrix of the map $\Gamma^{-1}_{t,x}: V(0;t,x,v)\rightarrow v$ (refer to Lemma \ref{lem2}) such that
\begin{equation*}
n_{f}(t,x) =\int_{\mathbb{R}^{3}} e^{3\kappa t}f_{0}(X(0;t,x,v),V(0;t,x,v)) {\rm{d}}v\leq 2 \|f_{0}\|_{L^{1}_{v}(L^{\infty}_{x})},
\end{equation*}
as long as the following uniform Lipschitz regularity of $u$ holds
\begin{equation}\label{nblul1inft}
\int_{0}^{t}\kappa\|\nabla u(\tau)\|_{L^{\infty}}{\rm{d}}\tau\ll 1.
\end{equation}

However, it is difficult to enclose the a-priori estimate \eqref{nblul1inft} since the techniques of maximal regularity estimates used in \cite{HanKwan2020,Ertzbischoff2021,HanK} can not be applied to the momentum equations $\eqref{NSV}_{2}$ containing the degenerate term $\rho u_{t}$ near the vacuum. To overcome this difficulty, we construct some new functional inequalities and established time-weighted estimates for higher order norms of $u$. This method is inspired by the works \cite{Danchin2019,He2021} for the pure inhomogeneous Navier-Stokes equations and allows us to capture some additional dissipation structures due to fluid-particle interactions. To be more specific, given the basic energy inequality \eqref{enrgineq}, we improve the estimates of $D(t)$ in Lemma \ref{lem4} and obtain a new functional inequality
\begin{equation}\label{dsptineq}
\frac{{\rm{d}}}{{\rm{d}}t} D(t)+\mathcal{E}(t)\leq g_{1}(t)D(t),
\end{equation}
where $g_{1}(t)$ is uniformly integrable, and $\mathcal{E}(t)$ is the high-order dissipation
\begin{equation}\label{dsdspt}
\mathcal{E}(t):=\int_{\mathbb{R}^{3}} \rho|u_{t}|^{2} {\rm{d}}x+\int_{\mathbb{R}^{3}\times\mathbb{R}^{3}}\kappa^{2}|u-v|^{2}f{\rm{d}}v{\rm{d}}x.
\end{equation}
Based on \eqref{dsptineq}, it is natural to enhance the estimate of $\mathcal{E}(t)$.  In  Lemma \ref{lem5}, we further obtain another new functional inequality
\begin{equation}\label{dsdsptineq}
\frac{{\rm{d}}}{{\rm{d}}t} \mathcal{E}(t)+\int_{\mathbb{R}^{3}}\mu|\nabla u_{t}|^{2}{\rm{d}}x+\int_{\mathbb{R}^{3}}\kappa n_{f}|u_{t}|^{2}{\rm{d}}x\leq g_{2}(t) \mathcal{E}(t)+F(t),
\end{equation}
for some uniformly integrable functions $g_{2}(t)$ and $F(t)$. In particular, the new inequalities \eqref{dsptineq} and \eqref{dsdsptineq}, together with the Stokes structure in $\eqref{NSV}_{2}$, implies the higher order estimates of $(u,P)$. With the aid of careful initial layer analysis on \eqref{dsptineq} and \eqref{dsdsptineq}, we obtain the short-time integration part of \eqref{nblul1inft}. Meanwhile, the large-time integration part of \eqref{nblul1inft} relies on the improvement of the exponential time-decay rate for $E(t)$ (refer to Lemma \ref{lem3}) to both $D(t)$ and $\mathcal{E}(t)$. Based on these uniform estimates and a bootstrap argument, we enclose \eqref{nblul1inft} and thence deduce the higher order regularities of $(\rho,u,f)$ and the decay estimates \eqref{result3}.

\vspace{2mm}

The rest of this paper is organized as follows. In Section \ref{secclb}, we devote to the exponential time-decay rate of the energy $E(t)$ under the additional bound \eqref{nfl1h} of $n_{f}$. In Section \ref{secape}, we justify the uniform bound \eqref{nfl1h}, enclose the a-priori Lipschitz regularity \eqref{nblul1inft} of $u$ and further obtain the higher order estimates and large-time behaviors of the solution. The main results, Theorems \ref{main1} and \ref{main2}, are proved in Section \ref{secprf}. Some definitions and lemmas are presented in \nameref{appendix}.

\section{Conditional large-time behavior}\label{secclb}

In this section, we prove the exponential stability of solutions to the Cauchy problem \eqref{NSV}--\eqref{inda} under the additional condition that the $L^{\frac{3}{2}}$-norm of the macroscopical density $n_{f}$ is uniformly bounded in time. Compared with the use of the Poincar\'{e} inequality in the spatial periodic case \cite{Choi2015}, our proof relies on the integrability of densities with vacuum.

\begin{proposition}\label{prop1}
For given time $T>0$, let $(\rho,u,f)$ be a strong solution to the Cauchy problem \eqref{NSV}--\eqref{inda} for $t\in(0,T]$. Suppose that there exist some constant $C_{*}>0$ independent of $T$ such that
\begin{equation}\label{nf32}
\sup_{t\in(0,T]}\|n_{f}(t)\|_{L^{\frac{3}{2}}}\leq C_{*}.
\end{equation}
Then it holds for $t\in(0,T]$ that
\begin{equation}\label{exp1}
E(t)\leq e^{-2\alpha t}E_{0}
\end{equation}
with the constant
\begin{equation}\label{alp}
\alpha:=\frac{\min\{\mu,\kappa\}}{2+C_{S}(\|\rho_{0}\|_{L^\frac{3}{2}}+2C_{*})}>0,
\end{equation}
where $C_{S}>0$ denotes the Sobolev constant such that $\|\cdot\|_{L^{6}}^{2}\leq C_{S}\|\nabla\cdot\|_{L^{2}}^{2}$ holds.
\end{proposition}

\begin{proof}
In view of \eqref{NSV}, one can obtain the energy inequality \eqref{enrgineq} for any $t\in(0,T]$. We claim that
\begin{equation}\label{core}
\int_{\mathbb{R}^{3}}\mu|\nabla u|^{2}{\rm{d}}x+\int_{\mathbb{R}^{3}\times\mathbb{R}^{3}}\kappa|u-v|^{2}f{\rm{d}}v{\rm{d}}x\geq 2\alpha E(t)
\end{equation}
with $\alpha$ given in \eqref{alp}. Based on \eqref{core}, we are able to employ the Gr\"{o}nwall inequality to \eqref{enrgineq} and then derive the exponential stability \eqref{exp1}.

Indeed, by the incompressible condition $\eqref{NSV}_{3}$ and the transport nature of $\eqref{NSV}_{1}$, a direct computation (cf. \cite{Lions}) deduces for $t\in(0,T]$ that
\begin{equation}\label{esro}
\|\rho\|_{L^{\frac{3}{2}}}=\|\rho_{0}\|_{L^{\frac{3}{2}}},
\end{equation}
which together with the Sobolev inequality yields
\begin{equation}\label{esro1}
\int_{\mathbb{R}^{3}}\rho|u|^{2}{\rm{d}}x\leq  \|\rho\|_{L^{\frac{3}{2}}}\|u\|_{L^{6}}^{2}\leq C_{S}\|\rho_{0}\|_{L^{\frac{3}{2}}}\|\nabla u\|_{L^{2}}^{2}.
\end{equation}
Similarly, one has
\begin{equation}\begin{aligned}
\int_{\mathbb{R}^{3}\times\mathbb{R}^{3}}|u|^{2}f  {\rm{d}}v{\rm{d}}x\leq C_{S}\|n_{f}\|_{L^\frac{3}{2}}\|\nabla u\|_{L^{2}}^{2}.\label{esro2}
\end{aligned}\end{equation}
Therefore, by \eqref{esro}--\eqref{esro2} and $\kappa\geq1$, we obtain
\begin{equation*}
\begin{aligned}
E(t)&\leq\frac{1}{2}\int_{\mathbb{R}^{3}}\rho|u|^{2}{\rm{d}}x+\int_{\mathbb{R}^{3}\times\mathbb{R}^{3}}(|u|^{2}+|u-v|^{2})f  {\rm{d}}v{\rm{d}}x\\
&\leq(2\alpha)^{-1}\left(\int_{\mathbb{R}^{3}}\mu|\nabla u|^{2}  {\rm{d}}x+\int_{\mathbb{R}^{3}\times\mathbb{R}^{3}}\kappa|u-v|^{2}f {\rm{d}}v{\rm{d}}x\right).
\end{aligned}
\end{equation*}
This gives rise to \eqref{core}, and thus the proof is completed.
\end{proof}

\section{Uniform a-priori estimates}\label{secape}

To begin with, we state the local existence and uniqueness of strong solutions to the Cauchy problem \eqref{NSV}--\eqref{inda} with vacuum.

\begin{proposition}\label{prop2}
Suppose that the initial data $(\rho_{0},u_{0},f_{0})$ satisfies \eqref{indacd}, then there exists a small time $T_{0}>0$ such that the Cauchy problem \eqref{NSV}--\eqref{inda} admits a unique strong solution $(\rho,u,f)$ on $(0,T_{0}]\times\mathbb{R}^{3}\times\mathbb{R}^{3}$ fulfilling \eqref{result1}.
\end{proposition}

The construction of the local-in-time solution is based on the convergence of approximated solutions away from vacuum as in \cite{Baranger2006,Lu2019,Cho2004,Choi2022}, and here we omit the details for brevity. The proof of the uniqueness is postponed to Section \ref{secprf}.

\subsection{The bootstrap}

To extend the local solution globally in time, we establish the following uniform a-priori bound.

\begin{proposition}\label{prop3}
There exists a constant $\varepsilon_{0}\in(0,1)$ depending only on $\|\rho_{0}\|_{L^\frac{3}{2}\cap L^{\infty}}$, $\|f_{0}\|_{L^{1}_{x,v}}$, $\|(1+|v|^{2})f_{0}\|_{L_{v}^{1}(L_{x}^{\infty})}$ and $M$ such that if \eqref{inenrgcd} holds, and $(\rho,u,f)$ is a smooth solution to \eqref{NSV} satisfying
\begin{equation}\label{priori}
\int_{0}^T(\mu^{-3}\|\nabla u(t)\|_{L^{2}}^{4}+\kappa\|u(t)\|_{L^{\infty}}+20\kappa\|\nabla u(t)\|_{L^{\infty}}) {\rm{d}}t\leq 1,
\end{equation}
then it holds the following estimate that
\begin{equation}\label{clspri}
\int_{0}^T(\mu^{-3}\|\nabla u(t)\|_{L^{2}}^{4}+\kappa\|u(t)\|_{L^{\infty}}+20\kappa\|\nabla u(t)\|_{L^{\infty}}){\rm{d}}t\leq \frac{1}{2}.
\end{equation}
\end{proposition}

Before giving the proof of Proposition \ref{prop3}, let us show some necessary uniform bounds on $(\rho,u,f)$. By standard maximal principles for the mass conservation equation $\eqref{NSV}_{1}$ and the Vlasov equation $\eqref{NSV}_{4}$, one has the following lemma.

\begin{lemma}\label{lem1}
Let $T>0$ and $(\rho,u,f)$ be the smooth solution to \eqref{NSV} for $(t,x,v)\in(0,T]\times\mathbb{R}^{3}\times\mathbb{R}^{3}$ with the initial data $(\rho_{0},u_{0},f_{0})$ satisfying \eqref{indacd}. Then under the condition \eqref{priori}, it holds
\begin{equation}\label{rhoinfty}
0\leq\rho(t,x)\leq\|\rho_{0}\|_{L^{\infty}},~~ 0\leq f(t,x,v)\leq e^{3\kappa T}\|f_{0}\|_{L^{\infty}_{x,v}}.
\end{equation}
\end{lemma}

In order to justify the key condition \eqref{nf32}, we establish the uniform estimates on the moments $n_{f}$, $j_{f}$ and $e_{f}$.

\begin{lemma}\label{lem2}
Let $T>0$ and $(\rho,u,f)$ be the smooth solution to \eqref{NSV} for $t\in(0,T]$ with the initial data $(\rho_{0},u_{0},f_{0})$ satisfying \eqref{indacd}. Then under the condition \eqref{priori}, it holds for $p\geq 1$ that
\begin{equation}\label{nf}
\sup_{t\in(0,T]}\|n_{f}(t)\|_{L^p}\leq 2\|f_{0}\|_{L^{1}_{x,v}}^\frac{1}{p}\|f_{0}\|_{L^{1}_{v}(L^{\infty}_{x})}^{1-\frac{1}{p}}.
\end{equation}
Furthermore, one has the following uniform upper bounds
\begin{align}
\label{jf}&\sup_{t\in(0,T]}\|j_{f}(t)\|_{L^{\infty}}\leq 2
\|(1+|v|)f_{0}\|_{L^{1}_{v}(L^{\infty}_{x})},\\
\label{ef}&\sup_{t\in(0,T]}\|e_{f}(t)\|_{L^{\infty}}\leq 2\|(1+|v|^{2})f_{0}\|_{L^{1}_{v}(L^{\infty}_{x})}.
\end{align}
\end{lemma}

\begin{proof}
The proof follows similar arguments as in \cite{Hofer,HanKwan2020,Ertzbischoff}. Let the characteristic curves $X(\tau;t,x,v)$ and $V(\tau;t,x,v)$ be given by \eqref{chrcrv}. Integrating $\eqref{NSV}_{4}$ over $[0,t]\times \mathbb{R}^{3}\times\mathbb{R}^{3}$ for $t\in(0,T]$, we get
\begin{equation}\label{nf1}
\|n_{f}\|_{L^{1}}=\|f_{0}\|_{L^{1}_{x,v}}.
\end{equation}
According to the definition of $n_{f}$ and the formula \eqref{dstrbt} of $f$, it holds
\begin{equation}\label{nfrep}
\begin{aligned}
n_{f}(t,x)&=e^{3\kappa t}\int_{\mathbb{R}^{3}} f_{0}(X(0;t,x,v),V(0;t,x,v)){\rm{d}}v\\
&=e^{3\kappa t}\int_{\mathbb{R}^{3}} f_{0}(X(0;t,x,\Gamma_{t,x}^{-1}(w)),w)\left|{\rm{det}}\, D_{v}\Gamma_{t,x}(\Gamma_{t,x}^{-1}(w))\right|^{-1}{\rm{d}}w
\end{aligned}
\end{equation}
with the mapping $\Gamma_{t,x}:v\mapsto w=V(0;t,x,v)$. To control \eqref{nfrep}, one needs to justify that $\Gamma_{t,x}$ is $C^{1}$-diffeomorphism. Based on \eqref{chrcrv}, one deduces that
\begin{equation}\label{soluv}
e^{\kappa t}v-V(0;t,x,v)=\int_{0}^te^{\kappa \tau}\kappa u(\tau,X(\tau;t,x,v)){\rm{d}}\tau.
\end{equation}
Thence taking the derivative of \eqref{soluv} with respect to $v$, we have
\begin{equation}\label{maomao}
e^{\kappa t}{\rm{Id}}-D_{v}\Gamma_{t,x}(v)=\int_{0}^te^{\kappa\tau}\kappa\nabla u(\tau,X(\tau;t,x,v)D_{v}X(\tau;t,x,v){\rm{d}}\tau,
\end{equation}
This requires a bound of $\|D_{v}X(\tau)\|_{L^{\infty}_{x,v}}$. To achieve it, we differentiate \eqref{chrcrv} with respect to $(x,v)$ to obtain
\begin{equation*}
\frac{{\rm{d}}}{{\rm{d}}\tau} D_{x,v}Z(\tau;t,x,v)=D_{x,v}W(\tau, Z(\tau;t,x,v))\cdot  D_{x,v}Z(\tau;t,x,v),
\end{equation*}
where $Z=(X,V)$ and $W=(V,\kappa(u-V))$. A use of Gr\"onwall's inequality yields
\begin{equation*}
\|D_{x,v}Z(\tau)\|_{L^{\infty}_{x,v}}\leq \| D_{x,v}Z(t)\|_{L^{\infty}_{x,v}} \exp\left(\int^{t}_{\tau}\|D_{x,v}W\|_{L^{\infty}_{x,v}}{\rm{d}}s\right).
\end{equation*}
Since $D_{x,v}Z(t)={\rm Id}$ and $\kappa\geq1$, it implies that for $0\leq\tau\leq t\leq T$,
\begin{equation}\label{Xvinfty}
\|(\nabla_{v}X,\nabla_{v}V)(\tau)\|_{L^{\infty}_{x,v}}\leq \exp{\left(\kappa(t-\tau)+ \kappa\int_{\tau}^{t}\|\nabla u\|_{L^{\infty}}{\rm{d}}s\right)}.
\end{equation}
By \eqref{priori}, \eqref{maomao} and \eqref{Xvinfty}, we get
\begin{equation*}
\begin{aligned}
\|e^{-\kappa t}D_{v}\Gamma_{t,x}(v)-{\rm{Id}}\|_{L^{\infty}_{x,v}}&\leq\int_{0}^te^{-\kappa(t-\tau)}\kappa\|\nabla u\|_{L^{\infty}}\|D_{v}X(\tau;t,x,v)\|_{L^{\infty}_{x,v}}{\rm{d}}\tau\\
&\leq\int_{0}^t\kappa\|\nabla u\|_{L^{\infty}}{\rm{d}}\tau\,\exp{\left(\int_{0}^t\kappa\|\nabla u\|_{L^{\infty}}{\rm{d}}\tau\right)}\\
&\leq \frac{1}{20}e^{\frac{1}{20}}\\&\leq\frac{1}{9}.
\end{aligned}
\end{equation*}
It then follows from Lemma \ref{IFthm} with $\phi=e^{-\kappa t}\Gamma_{t,x}$ that
\begin{equation}
|{\rm{det}}\, D_{v}\Gamma_{t,x}(v)|=e^{3\kappa t} |{\rm{det}}\, \nabla \phi|\geq \frac{1}{2}e^{3\kappa t},\label{diffeomorphism}
\end{equation}
which combined with \eqref{nfrep} yields
\begin{equation}\label{nf2}
\begin{aligned}
n_{f}(t,x)&\leq e^{3\kappa t}\int_{\mathbb{R}^{3}} f_{0}(X(0;t,x,\Gamma_{t,x}^{-1}(w)),w)\cdot 2e^{-3\kappa t}{\rm{d}}w\\
&\leq 2\int_{\mathbb{R}^{3}} f_{0}(X(0;t,x,\Gamma_{t,x}^{-1}(w)),w){\rm{d}}w\\&\leq 2\|f_{0}\|_{L^{1}_{v}(L^{\infty}_{x})}.
\end{aligned}
\end{equation}
Hence,  by using \eqref{nf1}, \eqref{nf2} and H\"{o}lder's inequality, we derive the estimate \eqref{nf} for $n_{f}$.

As for $j_{f}$, due to the formula \eqref{maomao} and the bound \eqref{diffeomorphism}, $j_{f}$ fulfills the following description
\begin{equation}\label{jflinfty}
\begin{aligned}
j_{f}(t,x)&=e^{3\kappa t}\int_{\mathbb{R}^{3}}\Gamma_{t,x}^{-1}(w)f_{0}(X(0;t,x,\Gamma_{t,x}^{-1}(w)),w)\left|{\rm{det}}\, D_{v}\Gamma_{t,x}(\Gamma_{t,x}^{-1}(w))\right|^{-1}{\rm{d}}w\\
&\leq  e^{3\kappa t}\int_{\mathbb{R}^{3}}\left|\Gamma_{t,x}^{-1}(w)\right|f_{0}(X(0;t,x,\Gamma_{t,x}^{-1}(w)),w)\cdot 2e^{-3\kappa t}{\rm{d}}w\\
&\leq 2\int_{\mathbb{R}^{3}} e^{-\kappa t}wf_{0}(X(0;t,x,\Gamma_{t,x}^{-1}(w)),w){\rm{d}}w\\
&\quad +2\int_{0}^t\int_{\mathbb{R}^{3}} e^{-\kappa (t-\tau)}\kappa\|u\|_{L^{\infty}}f_{0}(X(0;t,x,\Gamma_{t,x}^{-1}(w)),w){\rm{d}}w{\rm{d}}\tau ,
\end{aligned}
\end{equation}
and this as well as \eqref{indacd} and \eqref{priori} leads to the estimate \eqref{jf}.

Similarly, it holds for $e_{f}$ that
\begin{equation}\label{eflinfty}
\begin{aligned}
e_{f}(t,x)&=\frac{1}{2}e^{3\kappa t}\int_{\mathbb{R}^{3}}|\Gamma_{t,x}^{-1}(w)|^{2}f_{0}(X(0;t,x,\Gamma_{t,x}^{-1}(w)),w)\left|{\rm{det}}\,D_{v}\Gamma_{t,x}(\Gamma_{t,x}^{-1}(w))\right|^{-1}{\rm{d}}w\\
&\leq \int_{\mathbb{R}^{3}} \left|e^{-\kappa t}w+\int_{0}^t e^{-\kappa (t-\tau)}\kappa u(\tau,X(\tau;t,x,\Gamma_{t,x}^{-1}(w))){\rm{d}}\tau\right|^{2}f_{0}(X(0;t,x,\Gamma_{t,x}^{-1}(w)),w){\rm{d}}w\\
&\leq 2e^{-2\kappa t}\||v|^{2}f_{0}\|_{L^{1}_{v}(L^{\infty}_{x})}+2\left|\int_{0}^t \kappa\|u\|_{L^{\infty}}{\rm{d}}\tau\ \right|^{2}\|f_{0}\|_{L^{1}_{v}(L_{x}^{\infty})}\\
&\leq 2\||v|^{2}f_{0}\|_{L^{1}_{v}(L^{\infty}_{x})}+2\|f_{0}\|_{L^{1}_{v}(L_{x}^{\infty})}.
\end{aligned}\end{equation}
Therefore, we complete the proof of Lemma \ref{lem2}.
\end{proof}

Next, combining Lemma \ref{lem2} with Proposition \ref{prop1}, we have the following exponential stability.

\begin{lemma}\label{lem3}
Let $T>0$ and $(\rho,u,f)$ be the smooth solution to \eqref{NSV} for $t\in(0,T]$ with the initial data $(\rho_{0},u_{0},f_{0})$ satisfying \eqref{indacd}. Then under the condition \eqref{priori}, it holds
\begin{equation}\label{eptd}
\sup_{t\in(0,T]}e^{\alpha_{1} t}E(t)+\frac{1}{2}\int_{0}^Te^{\alpha_{1} t} D(t){\rm{d}}t\leq E_{0}
\end{equation}
with the constant
\begin{equation}\label{alpha1}
\alpha_{1}:=\frac{\min\{\mu,1\}}{2+C_{S}(\|\rho_{0}\|_{L^\frac{3}{2}}+4\|f_{0}\|_{L^{1}_{x,v}}^\frac{2}{3}\|f_{0}\|_{L^{1}_{v}(L^{\infty}_{x})}^{\frac{1}{3}})}>0,
\end{equation}
where $E(t)$, $D(t)$ and $E_{0}$ are given by \eqref{enrg}, \eqref{dspt} and \eqref{inenrg}, respectively, and $C_{S}$ is the Sobolev constant.
\end{lemma}

\begin{proof}
Substituting the bound \eqref{nf} with $p=\frac{3}{2}$ into \eqref{alp} and gathering \eqref{enrgineq}, \eqref{core} and $\kappa\geq 1$, one can verify that
\begin{equation}\label{gr}
\frac{{\rm{d}}}{{\rm{d}}t}E(t)+\alpha_{1} E(t)+\frac{1}{2}D(t)\leq 0.
\end{equation}
Multiplying two sides of \eqref{gr} by $e^{\alpha_{1} t}$ and then integrating the resulting inequality over $[0,t]$ for all $t\in(0,T]$, we conclude \eqref{eptd} and complete the proof.
\end{proof}

Then, we provide the following key lemma concerning the uniform estimate of the dissipation $D(t)$.

\begin{lemma}\label{lem4}
Let $T>0$ and $(\rho,u,f)$ be the smooth solution to \eqref{NSV} for $t\in(0,T]$ with the initial data $(\rho_{0},u_{0},f_{0})$ satisfying \eqref{indacd}. Then under the condition \eqref{priori}, it holds
\begin{equation}\label{dsin0}
\sup_{t\in(0,T]}D(t)+\int_{0}^T\left(\mathcal{E}(t)+\mu^{2} \|\nabla^{2}u\|_{L^{2}}^{2}+\|\nabla P\|_{L^{2}}^{2}\right){\rm{d}}t\leq C_{3}(\mu+\kappa)M.
\end{equation}
Moreover, for any $T>1$ and $\beta_{1}\in[0,1]$, we have the short-time weighted estimates
\begin{equation}\label{dsin}
\begin{aligned}
&\sup_{t\in(0,1]}t^{\beta_{1}}D(t)+\int_{0}^{1}t^{\beta_{1}}\left(\mathcal{E}(t)+\mu^{2} \|\nabla^{2}u\|_{L^{2}}^{2}+\|\nabla P\|_{L^{2}}^{2}\right){\rm{d}}t\\
&\leq C_{3}(\mu+\kappa)^{1-\beta_{1}}E_{0}^{\beta_{1}} M^{1-\beta_{1}},
\end{aligned}
\end{equation}
and the large-time exponential stability
\begin{equation}\label{exdsin}
\begin{aligned}
\sup_{t\in[1,T]}e^{\alpha_{1}t}D(t)+\int_{1}^Te^{\alpha_{1}t}\left(\mathcal{E}(t)+\mu^{2} \|\nabla^{2}u\|_{L^{2}}^{2}+\|\nabla P\|_{L^{2}}^{2}\right){\rm{d}}t\leq C_{3}E_{0}.
\end{aligned}
\end{equation}
Here the functionals $D(t)$ and $\mathcal{E}(t)$ are defined by \eqref{dspt} and \eqref{dsdspt}, the constants $E_{0}$, $M$ and $\alpha_{1}$ are given by \eqref{inenrg}, \eqref{indspt} and \eqref{alpha1}, respectively, and $C_{3}>0$ is a constant depending only on $\|\rho_{0}\|_{L^{\infty}}$ and $\|f_{0}\|_{L^{1}_{v}(L^{\infty}_{x})}$.
\end{lemma}

\begin{proof}
We first show the uniform estimate \eqref{dsin0}. Multiplying $\eqref{NSV}_{2}$ by $u_{t}$ and then integrating by parts, one has
\begin{equation}\label{dpes01}
\frac{1}{2}\frac{{\rm{d}}}{{\rm{d}}t}\int_{\mathbb{R}^{3}}\mu|\nabla u|^{2}{\rm{d}}x+\int_{\mathbb{R}^{3}}\rho|u_{t}|^{2}{\rm{d}}x=-\int_{\mathbb{R}^{3}} u_{t}\cdot \rho u\cdot\nabla u {\rm{d}}x-\int_{\mathbb{R}^{3}\times\mathbb{R}^{3}}\kappa u_{t}\cdot(u-v)f{\rm{d}}v{\rm{d}}x.
\end{equation}
It follows from $\eqref{NSV}_{3}$ that
\begin{equation}\label{dpes02}
\begin{aligned}
&\int_{\mathbb{R}^{3}\times\mathbb{R}^{3}} u_{t}\cdot(u-v)f{\rm{d}}v{\rm{d}}x\\
&=\frac{1}{2}\frac{{\rm{d}}}{{\rm{d}}t} \int_{\mathbb{R}^{3}\times\mathbb{R}^{3}} (|u|^{2}-2u\cdot v) f {\rm{d}}v{\rm{d}}x-\frac{1}{2}\int_{\mathbb{R}^{3}\times\mathbb{R}^{3}} (|u|^{2}-2u\cdot v)f_{t}{\rm{d}}v{\rm{d}}x\\
&=\frac{1}{2}\frac{{\rm{d}}}{{\rm{d}}t} \int_{\mathbb{R}^{3}\times\mathbb{R}^{3}} (|u|^{2}-2u\cdot v) f {\rm{d}}v{\rm{d}}x\\
&\quad+\frac{1}{2}\int_{\mathbb{R}^{3}\times\mathbb{R}^{3}}(|u|^{2}-2u\cdot v)(v\cdot\nabla_{x}f+{\rm{div}}(\kappa(u-v)f)){\rm{d}}v{\rm{d}}x\\
&=\frac{1}{2}\frac{{\rm{d}}}{{\rm{d}}t} \int_{\mathbb{R}^{3}\times\mathbb{R}^{3}} (|u|^{2}-2u\cdot v) f {\rm{d}}v{\rm{d}}x-\int_{\mathbb{R}^{3}\times\mathbb{R}^{3}} (v\cdot\nabla_{x} u-\kappa u)\cdot (u-v)f{\rm{d}}v{\rm{d}}x\\
&=\frac{1}{2}\frac{{\rm{d}}}{{\rm{d}}t} \int_{\mathbb{R}^{3}\times\mathbb{R}^{3}} |u-v|^{2}f {\rm{d}}v{\rm{d}}x-\int_{\mathbb{R}^{3}\times\mathbb{R}^{3}} (v\cdot\nabla_{x}u)\cdot (u-v)f {\rm{d}}v{\rm{d}}x+\kappa\int_{\mathbb{R}^{3}\times\mathbb{R}^{3}} |u-v|^{2}f {\rm{d}}v{\rm{d}}x,
\end{aligned}\end{equation}
where we noticed that the energy of the Vlasov equation $\eqref{NSV}_{3}$ satisfies
\begin{equation}\label{dpes03}
\frac{1}{2}\frac{{\rm{d}}}{{\rm{d}}t}\int_{\mathbb{R}^{3}\times\mathbb{R}^{3}} |v|^{2} f{\rm{d}}v{\rm{d}}x=\int_{\mathbb{R}^{3}\times\mathbb{R}^{3}} \kappa v\cdot(u-v)f {\rm{d}}v{\rm{d}}x.
\end{equation}
The combination of \eqref{dpes01} and \eqref{dpes02} yields
\begin{equation}\label{dpes}
\frac{1}{2}\frac{{\rm{d}}}{{\rm{d}}t}D(t)+\mathcal{E}(t)=-\int_{\mathbb{R}^{3}} u_{t} \cdot\rho u\cdot\nabla u{\rm{d}}x+\int_{\mathbb{R}^{3}\times\mathbb{R}^{3}}\kappa (v\cdot\nabla_{x}u)\cdot (u-v)f {\rm{d}}v{\rm{d}}x,
\end{equation}
where we recalled
\begin{equation*}
\mathcal{E}(t):=\int_{\mathbb{R}^{3}} \rho|u_{t}|^{2} {\rm{d}}x+\int_{\mathbb{R}^{3}\times\mathbb{R}^{3}}\kappa^{2}|u-v|^{2}f{\rm{d}}v{\rm{d}}x.
\end{equation*}
To estimate the terms on the right-hand side of \eqref{dpes}, we apply \eqref{rhoinfty}  and the Gagliardo-Nirenberg inequality \eqref{GNSine} to obtain
\begin{equation}\label{mmmmm1}
\begin{aligned}
\left|\int_{\mathbb{R}^{3}} u_{t} \cdot\rho u\cdot\nabla u{\rm{d}}x\right|&\leq \|\rho_{0}\|_{L^{\infty}}^\frac{1}{2}\|\sqrt{\rho}u_{t}\|_{L^{2}}\|u\|_{L^{6}}\|\nabla u\|_{L^{3}}\\
&\leq C\|\sqrt{\rho}u_{t}\|_{L^{2}}\|\nabla u\|_{L^{2}}^{\frac{3}{2}}\|\nabla^{2}u\|_{L^{2}}^{\frac{1}{2}}\\
&\leq \eta \mu^{2} \|\nabla^{2}u\|_{L^{2}}^{2}+\eta\|\sqrt{\rho}u\|_{L^{2}}^{2}+ C_{\eta}\mu^{-2}\|\nabla u\|_{L^{2}}^{6},
\end{aligned}\end{equation}
where $\eta>0$ is a small constant to be determined later.  Similarly, it holds that
\begin{equation}\label{mmmmm2}
\begin{aligned}
&\kappa \left|\int_{\mathbb{R}^{3}\times\mathbb{R}^{3}} (v\cdot\nabla_{x}u)\cdot (u-v)f {\rm{d}}v{\rm{d}}x\right|\\
&\leq\kappa \left|\int_{\mathbb{R}^{3}\times\mathbb{R}^{3}} (v-u)\cdot\nabla_{x}u\cdot (u-v)f {\rm{d}}v{\rm{d}}x\right|+\kappa\left|\int_{\mathbb{R}^{3}\times\mathbb{R}^{3}} u\cdot\nabla_{x}u\cdot (u-v)f {\rm{d}}v{\rm{d}}x\right|\\
&\leq \kappa\|\nabla u\|_{L^{\infty}}\int_{\mathbb{R}^{3}\times\mathbb{R}^{3}} |u-v|^{2}f{\rm{d}}v{\rm{d}}x\\
&\quad+ C\kappa\|\nabla u\|_{L^{2}}^{\frac{3}{2}}\|\nabla^{2}u\|_{L^{2}}^{\frac{1}{2}}\left(\int_{\mathbb{R}^{3}\times\mathbb{R}^{3}} |u-v|^{2}f{\rm{d}}v{\rm{d}}x\right)^{\frac{1}{2}}\\
&\leq \left(\kappa\|\nabla u\|_{L^{\infty}}+\eta \kappa^{2}\right) \int_{\mathbb{R}^{3}\times\mathbb{R}^{3}} |u-v|^{2}f{\rm{d}}v{\rm{d}}x+ \eta \mu^{2} \|\nabla^{2}u\|_{L^{2}}^{2}+C_{\eta}\mu^{-2}\|\nabla u\|_{L^{2}}^{6}.
\end{aligned}\end{equation}
Substituting \eqref{mmmmm1} and \eqref{mmmmm2} into \eqref{dpes}  yields
\begin{equation}\label{dpesooo}
\frac{1}{2}\frac{{\rm{d}}}{{\rm{d}}t}D(t)+(1-\eta)\mathcal{E}(t)\leq C_{\eta}\Big(\kappa\|\nabla u\|_{L^{\infty}}+\mu^{-3}\|\nabla u\|_{L^{2}}^{4}\Big) D(t)+2\eta \mu^{2} \|\nabla^{2}u\|_{L^{2}}^{2}.
\end{equation}
To handle the $L^{2}$-norm of $\nabla^{2} u$ involved in \eqref{dpesooo}, we apply the Stokes estimate (see Lemma \ref{SSineq}) to $\eqref{NSV}_{2}$, \eqref{nf2} and the Gagliardo-Nirenberg inequality \eqref{GNSine}, and then deduce that
\begin{equation*}
\begin{aligned}
&\mu^{2}\|\nabla^{2} u\|_{L^{2}}^{2}+\|\nabla P\|_{L^{2}}^{2}\\
&\leq C\Big\|\rho u_{t}+\rho u\cdot\nabla u+\int_{\mathbb{R}^{3}}\kappa(u-v)f {\rm{d}}v\Big\|_{L^{2}}^{2}\\
&\leq C\|\rho_{0}\|_{L^{\infty}}\|\sqrt{\rho}u_{t}\|_{L^{2}}^{2}+C\|\rho_{0}\|_{L^{\infty}}^{2}\|u\|_{L^{6}}^{2}\|\nabla u\|_{L^{3}}^{2}+C\kappa^{2}\|n_{f}\|_{L_{x}^{\infty}}\int_{\mathbb{R}^{3}\times\mathbb{R}^{3}}|u-v|^{2}f {\rm{d}}v{\rm{d}}x\\
&\leq \frac{\mu^{2}}{2}  \|\nabla^{2}u\|_{L^{2}}^{2}+C\|\sqrt{\rho}u_{t}\|_{L^{2}}^{2}+C\mu^{-2}\|\nabla u\|_{L^{2}}^{6}+C\kappa^{2} \|f_{0}\|_{L^{1}_{v}(L^{\infty}_{x})}\int_{\mathbb{R}^{3}\times\mathbb{R}^{3}}|u-v|^{2}f {\rm{d}}v{\rm{d}}x,
\end{aligned}
\end{equation*}
where we used the upper bound of $n_{f}$ in \eqref{nf}. Hence, it follows that
\begin{equation}\label{mmmmm3}
\mu^{2}\|\nabla^{2} u\|_{L^{2}}^{2}+\|\nabla P\|_{L^{2}}^{2}\leq C(1+\|f_{0}\|_{L^{1}_{v}(L^{\infty}_{x})})\mathcal{E}(t)+C\mu^{-3}\|\nabla u\|_{L^{2}}^{4} D(t).
\end{equation}
Multiplying \eqref{mmmmm3} by $3\eta$, adding the resulting inequality and $\eqref{dpesooo}$ together and then choosing
\begin{equation*}
\eta=\eta_{f_{0}}=\frac{1}{2(1+3C(1+\|f_{0}\|_{L^{1}_{v}(L^{\infty}_{x})}))},
\end{equation*}
we conclude that
\begin{equation}\label{mmmmm4}
\frac{{\rm{d}}}{{\rm{d}}t}D(t)+\mathcal{E}(t)+\eta_{f_{0}}\left(\mu^{2}\|\nabla^{2} u\|_{L^{2}}^{2}+\|\nabla P\|_{L^{2}}^{2}\right)\leq C\left(\kappa\|\nabla u\|_{L^{\infty}}+\mu^{-3}\|\nabla u\|_{L^{2}}^{4}\right) D(t).
\end{equation}
Taking advantage of the Gr\"{o}nwall inequality and \eqref{priori}, we obtain
\begin{equation}\label{dsin01}
\begin{aligned}
&\sup_{t\in(0,T]}D(t)+\int_{0}^T\left(\mathcal{E}(t)+\mu^{2}\|\nabla^{2}u\|_{L^{2}}^{2}+\|\nabla P\|_{L^{2}}^{2}\right){\rm{d}}t\\
&\leq \exp\left(C\int_{0}^T\left(\kappa\|\nabla u\|_{L^{\infty}}+\mu^{-3}\|\nabla u\|_{L^{2}}^{4}\right){\rm{d}}t\right)D(0)\\&\leq C(\mu+\kappa)M.
\end{aligned}
\end{equation}

Next, in the short-time case $t\in[0,1]$, we multiply \eqref{mmmmm4} by the time weight $t$ to have
\begin{equation*}
\begin{aligned}
&\frac{{\rm{d}}}{{\rm{d}}t}\big(tD(t)\big)+t\left(\mathcal{E}(t)+\eta_{f_{0}} (\mu^{2} \|\nabla^{2}u\|_{L^{2}}^{2}+\|\nabla P\|_{L^{2}}^{2})\right)\\
&\leq C(\kappa\|\nabla u\|_{L^{\infty}}+\mu^{-3} \|\nabla u\|_{L^{2}}^{4})\big(tD(t)\big)+D(t).
\end{aligned}
\end{equation*}
Since $D(t)\in L^{1}(0,T)$, there exists some sequence $t_{k}\rightarrow 0$ such that $\displaystyle\lim_{t_{k}\rightarrow 0} t_{k} D(t_k)=0$. Applying Gr\"onwall's inequality and \eqref{priori} on $[t_{k},1]$, letting $t_{k}\rightarrow 0$  and then  employing \eqref{eptd}, we obtain the time-weighted estimate
\begin{equation}\label{weightt}
\begin{aligned}
&\sup_{t\in(0,1]}tD(t)+\int_{0}^{1}t\left(\mathcal{E}(t)+\mu^{2} \|\nabla^{2}u\|_{L^{2}}^{2}+\|\nabla P\|_{L^{2}}^{2}\right){\rm{d}}t\\
&\leq C\int_{0}^{1}D(t){\rm{d}}t\\&\leq CE_{0}.
\end{aligned}
\end{equation}
Hence, the H\"older inequality between \eqref{dsin01} and \eqref{weightt} leads to \eqref{dsin}.

Finally, in the large-time case $t\in[1,T]$, multiplying \eqref{mmmmm4} by the exponential weight $e^{\alpha_{1} t}$, we have
\begin{equation*}
\begin{aligned}
&\frac{{\rm{d}}}{{\rm{d}}t}\big(e^{\alpha_{1} t}D(t)\big)+e^{\alpha_{1} t}\left(\mathcal{E}(t)+\eta_{f_{0}}\left(\mu^{2} \|\nabla^{2}u\|_{L^{2}}^{2}+\|\nabla P\|_{L^{2}}^{2}\right)\right)\\
&\leq C\left(\kappa\|\nabla u\|_{L^{\infty}}+ \mu^{-3} \|\nabla u\|_{L^{2}}^{4}\right)e^{\alpha_{1} t}D(t)+\alpha_{1} e^{\alpha_{1} t}D(t),
\end{aligned}
\end{equation*}
which, in a similar argument as above, yields that
\begin{equation*}
\begin{aligned}
&\sup_{t\in[1,T]}e^{\alpha_{1}t}D(t)+\int_{1}^Te^{\alpha_{1}t}\left(\mathcal{E}(t)+\frac{1}{2C}\mu^{2} \|\nabla^{2}u\|_{L^{2}}^{2}+\|\nabla P\|_{L^{2}}^{2}\right){\rm{d}}t\\
&\leq e^{\alpha_{1}}D(1)+C\alpha_{1} \int_{1}^T e^{\alpha_{1}t}D(t){\rm{d}}t\\&\leq CE_{0},
\end{aligned}
\end{equation*}
where we have applied \eqref{weightt} with $t=1$ and noticed that $\alpha_{1}\leq 1$.

Therefore, we complete the proof of Lemma \ref{lem4}.
\end{proof}

In addition, we establish the higher order weighted estimates of $u$ as follows.

\begin{lemma}\label{lem5}
Let $T>1$ and $(\rho,u,f)$ be the smooth solution to \eqref{NSV} for $t\in(0,T]$ with the initial data $(\rho_{0},u_{0},f_{0})$ satisfying \eqref{indacd}. Then under the condition \eqref{priori}, for any $\beta_{2}\in[0,1]$, one has the short-time estimates
\begin{equation}\label{twds}
\begin{aligned}
&\sup_{t\in(0,1]}t^{1+\beta_{2}}\mathcal{E}(t)+\int_{0}^{1}
t^{1+\beta_{2}}\left(\mu\|\nabla u_{t}(t)\|_{L^{2}}^{2}+ \kappa \|\sqrt{n_{f}}u_{t}(t)\|_{L^{2}}^{2}\right){\rm{d}}t\\
&\leq C_{4}\Big(\mu^{-5}(\mu+\kappa)E_{0} M+\mu^{-5}E_{0} +\mu^{-1}\kappa+E_{0}+1 \Big) (\mu+\kappa)^{1-\beta_{2}} E_{0}^{\beta_{2}} M^{1-\beta_{2}},
\end{aligned}
\end{equation}
and the large-time exponential stability
\begin{equation}\label{extwds}
\begin{aligned}
&\sup_{t\in[1,T]}e^{\alpha_{1} t}\mathcal{E}(t)+\int_{1}^T
e^{\alpha_{1} t}\left(\mu\|\nabla u_{t}(t)\|_{L^{2}}^{2}+ \kappa \|\sqrt{n_{f}}u_{t}(t)\|_{L^{2}}^{2}\right){\rm{d}}t\\
&\leq C_{4}\Big(\mu^{-5}(\mu+\kappa)E_{0} M+\mu^{-5}E_{0} +\mu^{-1}\kappa+E_{0}+1 \Big)  E_{0},
\end{aligned}
\end{equation}
where the functional $\mathcal{E}(t)$ is defined by \eqref{dsdspt}, the constants $E_{0}$, $M$ and $\alpha_{1}$ are given by \eqref{inenrg}, \eqref{indspt} and \eqref{alpha1}, respectively, and $C_{4}>0$ is a constant depending only on $\|\rho_{0}\|_{L^\frac{3}{2}\cap L^{\infty}}$, $\|f_{0}\|_{L^{1}_{x,v}}$ and $\|(1+|v|^{2}) f_{0}\|_{L_{v}^{1}(L_{x}^{\infty})}$.
\end{lemma}

\begin{proof}
Applying $\partial_{t}$ to $\eqref{NSV}_{2}$, we have
\begin{equation}\label{rutt}
\rho (u_{tt}+ u\cdot\nabla u_{t})+\nabla P_{t}-\mu\Delta u_{t}+\kappa n_{f} u_{t}=-\rho_{t}u_{t}-(\rho u)_{t}\cdot\nabla u-\int_{\mathbb{R}^{3}}\kappa(u-v)f_{t} {\rm{d}}v.
\end{equation}
Multiplying \eqref{rutt} by $u_{t}$ and integrating by parts, we deduce that
\begin{equation}\label{311}
\begin{aligned}
&\frac{1}{2}\frac{{\rm{d}}}{{\rm{d}}t}\int_{\mathbb{R}^{3}}\rho|u_{t}|^{2}{\rm{d}}x+\int_{\mathbb{R}^{3}}\mu|\nabla u_{t}|^{2}{\rm{d}}x+\int_{\mathbb{R}^{3}}\kappa n_{f}|u_{t}|^{2}{\rm{d}}x\\
&=-\int_{\mathbb{R}^{3}} (\rho_{t}u_{t}+(\rho u)_{t}\cdot\nabla u)\cdot u_{t} {\rm{d}}x-\int_{\mathbb{R}^{3}\times\mathbb{R}^{3}} \kappa u_{t}\cdot(u-v)f_{t} {\rm{d}}v{\rm{d}}x\\
&:=I_{1}+I_{2}.
\end{aligned}
\end{equation}
Below, we first deal with the term $I_{1}$ on the right-hand side of \eqref{311}. Making use of $\eqref{NSV}_{1}$, we further decompose $I_{1}$ as
\begin{equation*}
\begin{aligned}
I_{1}&=2\int_{\mathbb{R}^{3}}\rho (u\cdot \nabla u_{t})\cdot u_{t}{\rm{d}}x+\int_{\mathbb{R}^{3}}\rho (u_{t}\cdot\nabla u)\cdot u_{t}{\rm{d}}x+\int_{\mathbb{R}^{3}}\rho u\cdot\nabla(u\cdot\nabla u\cdot u_{t}){\rm{d}}x\\
&:=I_{11}+I_{12}+I_{13}.
\end{aligned}
\end{equation*}
It follows from \eqref{rhoinfty}, the Gagliardo-Nirenberg inequality \eqref{GNSine} and the Sobolev embedding $\dot{H}^{1}\hookrightarrow L^{6}$ that
\begin{equation*}
\begin{aligned}
I_{11}&\leq 2\|\rho\|_{L^{\infty}}^{\frac{1}{2}}\|\sqrt{\rho} u_{t}\|_{L^{3}}\|u\|_{L^{6}}\|\nabla u_{t}\|_{L^{2}}\\
&\leq C\|\rho_{0}\|_{L^{\infty}}^{\frac{3}{4}} \|\sqrt{\rho}u_{t}\|_{L^{2}}^\frac{1}{2}\|u_{t}\|_{L^{6}}^\frac{1}{2}\|\nabla u\|_{L^{2}}\|\nabla u_{t}\|_{L^{2}}\\
&\leq C\|\sqrt{\rho}u_{t}\|_{L^{2}}^{\frac{1}{2}}\|\nabla u_{t}\|_{L^{2}}^{\frac{3}{2}}\|\nabla u\|_{L^{2}}\\
&\leq \frac{\mu}{10}\|\nabla u_{t}\|_{L^{2}}^{2}+C\mu^{-3}\|\nabla u\|_{L^{2}}^{4}\|\sqrt{\rho}u_{t}\|_{L^{2}}^{2}.
\end{aligned}
\end{equation*}
Similarly, one has
\begin{equation*}
\begin{aligned}
I_{12}&\leq \|\rho\|_{L^{\infty}}^{\frac{1}{2}}\|\sqrt{\rho} u_{t}\|_{L^{3}}\|\nabla u\|_{L^{2}}\|u_{t}\|_{L^{6}}\\
&\leq \frac{\mu}{10}\|\nabla u_{t}\|_{L^{2}}^{2}+C\mu^{-3}\|\nabla u\|_{L^{2}}^{4}\|\sqrt{\rho}u_{t}\|_{L^{2}}^{2},
\end{aligned}
\end{equation*}
and
\begin{equation*}
\begin{aligned}
I_{13}&\leq \|\rho\|_{L^{\infty}}\|u\|_{L^{6}}\|\nabla(u\cdot\nabla u\cdot u_{t})\|_{L^\frac{6}{5}}\\
&\leq C\|\rho_{0}\|_{L^{\infty}}\|\nabla u\|_{L^{2}}^{2}\|\nabla u\|_{L^{6}}\|u_{t}\|_{L^{6}}\\
&\quad+C\|\rho_{0}\|_{L^{\infty}}\|\nabla u\|_{L^{2}}\|u\|_{L^{6}}\|\nabla^{2}u\|_{L^{2}}\|u_{t}\|_{L^{6}}\\
&\quad+C\|\rho_{0}\|_{L^{\infty}}\|\nabla u\|_{L^{2}}\|u\|_{L^{6}}\|\nabla u\|_{L^{6}}\|\nabla u_{t}\|_{L^{2}}\\
&\leq C\|\nabla u\|_{L^{2}}^{2}\|\nabla^{2}u\|_{L^{2}}\|\nabla u_{t}\|_{L^{2}}\\
&\leq \frac{\mu}{10}\|\nabla u_{t}(t)\|_{L^{2}}^{2}+C\mu^{-1}\|\nabla u(t)\|_{L^{2}}^{4}\|\nabla^{2}u(t)\|_{L^{2}}^{2}.
\end{aligned}
\end{equation*}
In addition, with the help of $\eqref{NSV}_{4}$, the second term on the right-hand side of \eqref{311} can be estimated by
\begin{equation}\label{313}
\begin{aligned}
I_{2}&=-\int_{\mathbb{R}^{3}\times\mathbb{R}^{3}} u_{t}\cdot(u-v)\big(v\cdot\nabla_{x}f+{\rm{div}}
_{v}(\kappa(u-v)f)\big){\rm{d}}v{\rm{d}}x\\
&\leq \int_{\mathbb{R}^{3}\times\mathbb{R}^{3}} u_{t}\cdot  (v\cdot\nabla u)f {\rm{d}}v{\rm{d}}x+\int_{\mathbb{R}^{3}\times\mathbb{R}^{3}} (v\cdot \nabla u_{t})\cdot (u-v) f{\rm{d}}v{\rm{d}}x+\int_{\mathbb{R}^{3}\times\mathbb{R}^{3}} u_{t}\cdot \kappa(u-v)f{\rm{d}}v{\rm{d}}x\\&:=I_{21}+I_{22}+I_{23}.
\end{aligned}
\end{equation}
One deduces from \eqref{nf} and \eqref{enrgineq} that
\begin{equation*}
\|j_{f}\|_{L^{2}}\leq\|n_{f}\|_{L^{\infty}}^\frac{1}{2}\left(\int_{\mathbb{R}^{3}\times\mathbb{R}^{3}}|v|^{2}f  {\rm{d}}v{\rm{d}}x\right)^\frac{1}{2}\leq 2\|f_{0}\|_{L^{1}_{v}(L^{\infty}_{x})}^\frac{1}{2}E_{0}^\frac{1}{2},
\end{equation*}
and hence it yields
\begin{equation*}
\begin{aligned}
I_{21}&\leq \|j_{f}\|_{L^{2}} \|u_{t}\|_{L^{6}}\|\nabla u\|_{L^{3}}\\
&\leq C\|f_{0}\|_{L^{1}_{v}(L^{\infty}_{x})}^{\frac{1}{2}}E_{0}^{\frac{1}{2}}\|\nabla u_{t}\|_{L^{2}}\|\nabla u\|_{L^{2}}^{\frac{1}{2}} \|\nabla^{2} u\|_{L^{2}}^{\frac{1}{2}}\\
&\leq \frac{\mu}{10} \|\nabla u_{t}\|_{L^{2}}^{2}+C E_{0}\mu^{-1}\|\nabla u\|_{L^{2}} \|\nabla^{2}u\|_{L^{2}}.
\end{aligned}
\end{equation*}
Note that the estimate \eqref{ef} ensures that
\begin{equation*}
\begin{aligned}
I_{22}&\leq \|\nabla u_{t}\|_{L^{2}}\|e_{f}\|_{L^{\infty}}^\frac{1}{2}\left(\int_{\mathbb{R}^{3}\times\mathbb{R}^{3}}|u-v|^{2}f{\rm{d}}v{\rm{d}}x\right)^\frac{1}{2}\\
&\leq \frac{\mu}{10} \|\nabla u_{t}\|_{L^{2}}^{2}+C\mu^{-1} \int_{\mathbb{R}^{3}\times\mathbb{R}^{3}}|u-v|^{2}f{\rm{d}}v{\rm{d}}x.
\end{aligned}
\end{equation*}
And it is direct to get
\begin{equation*}
\begin{aligned}
I_{23}&\leq \kappa\|\sqrt{n_{f}}u_{t}\|_{L^{2}}\left(\int_{\mathbb{R}^{3}\times\mathbb{R}^{3}}|u-v|^{2}f{\rm{d}}v{\rm{d}}x\right)^\frac{1}{2}\\
&\leq \frac{\kappa}{4}\|\sqrt{n_{f}}u_{t}\|_{L^{2}}^{2}+C\kappa \int_{\mathbb{R}^{3}\times\mathbb{R}^{3}}|u-v|^{2}f{\rm{d}}v{\rm{d}}x.
\end{aligned}
\end{equation*}
Substituting the above estimates of $I_{1}$ and $I_{2}$ into \eqref{311}, we derive
\begin{equation}\label{3111}
\begin{aligned}
&\frac{{\rm{d}}}{{\rm{d}}t}\int_{\mathbb{R}^{3}}\rho|u_{t}|^{2}{\rm{d}}x+\int_{\mathbb{R}^{3}}\mu|\nabla u_{t}|^{2}{\rm{d}}x+\int_{\mathbb{R}^{3}}\kappa n_{f}|u_{t}|^{2}{\rm{d}}x\\
&\leq C\mu^{-3}\|\nabla u\|_{L^{2}}^{4}\mathcal{E}(t)+C\mu^{-1}\|\nabla u\|_{L^{2}}^{4}\|\nabla^{2}u\|_{L^{2}}^{2}\\
&\quad+C E_{0} \mu^{-1}\|\nabla u\|_{L^{2}} \|\nabla^{2}u\|_{L^{2}}+C(\mu^{-1}+1)D(t).
\end{aligned}
\end{equation}
Here owing to \eqref{mmmmm3}, the second term on the right-hand side of \eqref{3111} can be controlled by
\begin{equation}\label{3111000}
C\mu^{-1}\|\nabla u\|_{L^{2}}^{4}\|\nabla^{2}u\|_{L^{2}}^{2}\leq C \mu^{-3}\|\nabla u\|_{L^{2}}^{4}\Big(\mathcal{E}(t)+\mu^{-5} D(t)^{3} \Big),
\end{equation}
and the third term has the following estimate
\begin{equation}\label{3111001}
\begin{aligned}
C E_{0} \mu^{-1}\|\nabla u\|_{L^{2}} \|\nabla^{2}u\|_{L^{2}}&\leq C E_{0} \mu^{-\frac{5}{2}}\sqrt{D(t)} \sqrt{\mathcal{E}(t)+\mu^{-5}D(t)^{3}}\\
&\leq C E_{0} \Big(\mathcal{E}(t)+ \mu^{-5} D(t)+\mu^{-5}D(t)^{2} \Big).
\end{aligned}
\end{equation}
Moreover, due to \eqref{ef} and the Young inequality, it holds
\begin{equation}\label{uvfkappa2}
\begin{aligned}
&\frac{{\rm{d}}}{{\rm{d}}t}\int_{\mathbb{R}^{3}\times\mathbb{R}^{3}}\kappa^{2}|u-v|^{2}f{\rm{d}}v{\rm{d}}x+2 \int_{\mathbb{R}^{3}\times\mathbb{R}^{3}} \kappa^{3}|u-v|^{2}f{\rm{d}}v{\rm{d}}x\\
&=2\int_{\mathbb{R}^{3}\times\mathbb{R}^{3}}\kappa^{2} u_{t}\cdot(u-v)f{\rm{d}}v{\rm{d}}x+2\int_{\mathbb{R}^{3}\times\mathbb{R}^{3}} \kappa^{2}v\cdot\nabla u\cdot(u-v)f{\rm{d}}v{\rm{d}}x\\
&\leq 2\kappa^{2} \|\sqrt{n_{f}}u_{t}\|_{L^{2}} \left(\int_{\mathbb{R}^{3}\times\mathbb{R}^{3}} |u-v|^{2}f{\rm{d}}v{\rm{d}}x\right)^{\frac{1}{2}}\\
&\quad+2\kappa^{2} \|e_{f}\|_{L^{\infty}}^{\frac{1}{2}}\|\nabla u\|_{L^{2}} \left( \int_{\mathbb{R}^{3}\times\mathbb{R}^{3}} |u-v|^{2}f{\rm{d}}v{\rm{d}}x\right)^{\frac{1}{2}}\\
&\leq 2\kappa\|\sqrt{n_{f}}u_{t}\|_{L^{2}}^{2}+2\kappa^{3}\int_{\mathbb{R}^{3}\times\mathbb{R}^{3}}  |u-v|^{2}f{\rm{d}}v{\rm{d}}x+C\mu^{-1}\kappa D(t).
\end{aligned}
\end{equation}
Plugging \eqref{3111000} and \eqref{3111001} into \eqref{3111}, together with \eqref{uvfkappa2}, we obtain that
\begin{equation}\label{ddtut}
\frac{{\rm{d}}}{{\rm{d}}t}\mathcal{E}(t)+\int_{\mathbb{R}^{3}}\mu|\nabla u_{t}|^{2}{\rm{d}}x+\int_{\mathbb{R}^{3}} \kappa n_{f}|u_{t}|^{2}{\rm{d}}x\leq C\mu^{-3}\|\nabla u\|_{L^{2}}^{4}\mathcal{E}(t)+C F(t),
\end{equation}
where
\begin{equation*}
F(t):=\mu^{-3}\|\nabla u\|_{L^{2}}^{4}\mu^{-5}D(t)^{3}+ E_{0}\mathcal{E}(t)+E_{0}\mu^{-5}D(t)^{2}+(1+\mu^{-1}\kappa +E_{0}\mu^{-5})D(t).
\end{equation*}
Multiplying \eqref{ddtut} by $t^{1+\beta_{2}}$ with $\beta_{2}\in[0,1]$ and $t\in[0,1]$, we get
\begin{equation}\label{314}
\begin{aligned}
&\frac{{\rm{d}}}{{\rm{d}}t}\Big(t^{1+\beta_{2}}\mathcal{E}(t)\Big)+t^{1+\beta_{2}} \left(\mu\|\nabla u_{t}\|_{L^{2}}^{2}+ \kappa \|\sqrt{n_{f}}u_{t}\|_{L^{2}}^{2}\right)\\
&\leq (1+\beta_{2}) t^{\beta_{2}} \mathcal{E}(t)+C\mu^{-3}\|\nabla u\|_{L^{2}}^{4} t^{1+\beta_{2}}\mathcal{E}(t)+C t^{1+\beta_{2}} F(t).
\end{aligned}
\end{equation}
Here the first term on the right-hand side of \eqref{314} corresponds to the dissipation in \eqref{dsin}, so there exists a sequence $t_{k}\rightarrow 0$ such that $\displaystyle\lim_{t_{k}\rightarrow 0}  t_k^{1+\beta_{2}} \mathcal{E}(t)=0$. By using the Gr\"onwall inequality to \eqref{314} with $t\in[t_{k},1]$ and then letting $t_{k}\rightarrow0$, we have
\begin{equation}\label{315}
\begin{aligned}
&\sup_{t\in(0,1]}t^{1+\beta_{2}}\mathcal{E}(t)+\int_{0}^{1}
t^{1+\beta_{2}}\left(\mu\|\nabla u_{t}\|_{L^{2}}^{2}+ \kappa \|\sqrt{n_{f}}u_{t}\|_{L^{2}}^{2}\right){\rm{d}}t\\
&\leq C\exp{\left(\int_{0}^{1}\mu^{-3}\|\nabla u\|_{L^{2}}^{4}{\rm{d}}t\right)} \Big( (1+\beta_{2}) \int_{0}^{1} t^{\beta_{2}} \mathcal{E}(t){\rm{d}}t+ \int_{0}^{1} t^{1+\beta_{2}}F(t){\rm{d}}t\Big).
\end{aligned}
\end{equation}
Thanks to \eqref{priori} and \eqref{dsin} with $\beta_{1}=\frac{1+\beta_{2}}{3}$ or $\beta_{1}=\beta_{2}$, it follows that
\begin{equation}\label{mmmmmmm}
\begin{aligned}
&\int_{0}^{1} t^{1+\beta_{2}}F(t){\rm{d}}t\\
&\leq \mu^{-5} \sup_{t\in(0,1]}\Big( t^{\frac{1+\beta_{2}}{3}} D(t) \Big)^{3}  \int_{0}^{1} \mu^{-3}\|\nabla u\|_{L^{2}}^{4}{\rm{d}}t+E_{0} \int_{0}^{1} t^{\beta_{2}} \mathcal{E}(t){\rm{d}}t\\
&\quad+E_{0}\mu^{-5}\sup_{t\in(0,1]} D(t)\sup_{t\in(0,1]} \Big(t^{\beta_{2}}D(t)\Big)+(1+\mu^{-1}\kappa +E_{0}\mu^{-5}) \sup_{t\in(0,1]} \Big(t^{\beta_{2}} D(t)\Big)\\
&\leq C\Big(\mu^{-5}(\mu+\kappa)E_{0} M +E_{0}+1+\mu^{-1}\kappa +\mu^{-5}E_{0}\Big) (\mu+\kappa)^{1-\beta_{2}} E_{0}^{\beta_{2}} M^{1-\beta_{2}}.
\end{aligned}
\end{equation}
By virtue of \eqref{priori}, \eqref{315} and \eqref{mmmmmmm}, we obtain the estimate \eqref{twds}.

Furthermore, multiplying \eqref{ddtut} by the exponential weight $e^{\alpha_{1} t}$ with $t\in[1,T]$ yields
\begin{equation}\label{3141}
\begin{aligned}
&\frac{{\rm{d}}}{{\rm{d}}t}\Big(e^{\alpha_{1} t}\mathcal{E}(t)\Big)+e^{\alpha_{1} t}\left(\mu\|\nabla u_{t}\|_{L^{2}}^{2}+ \kappa \|\sqrt{n_{f}}u_{t}\|_{L^{2}}^{2}\right)\\
&\leq \alpha_{1} e^{\alpha_{1} t}\mathcal{E}(t)+C\mu^{-3}\|\nabla u\|_{L^{2}}^{4}e^{\alpha_{1} t}\mathcal{E}(t)+Ce^{\alpha_{1} t}F(t).
\end{aligned}
\end{equation}
A similar computation as in \eqref{mmmmmmm}, one deduces from \eqref{dsin0}, \eqref{eptd} and \eqref{exdsin} that
\begin{equation*}
\begin{aligned}
\int_{1}^{T}e^{\alpha_{1} t}F(t){\rm{d}}t&\leq C\mu^{-5} \sup_{t\in[1,T]} D(t)\sup_{t\in[1,T]}\Big( e^{\alpha_{1} t} D(t) \Big)^{2} \int_{1}^{T} \mu^{-3}\|\nabla u\|_{L^{2}}^{4}{\rm{d}}t\\
&\quad+E_{0} \int_{1}^{T} e^{\alpha_{1} t} \mathcal{E}(t){\rm{d}}t+CE_{0}\mu^{-5} \sup_{t\in[1,T]} D(t)\int_{1}^{T}e^{\alpha_{1}t} D(t){\rm{d}}t\\
&\quad+(1+\mu^{-1}\kappa +E_{0}\mu^{-5})\int_{1}^{T} e^{\alpha_{1} t} D(t){\rm{d}}t\\
&\leq C\Big( \mu^{-5}(\mu+\kappa)E_{0}M +E_{0}+1+\mu^{-1}\kappa +\mu^{-5}E_{0}\Big)E_{0}.
\end{aligned}
\end{equation*}
Hence, it holds by the Gr\"{o}nwall inequality to \eqref{3141} on $[1,T]$, \eqref{priori} and \eqref{twds} with $t=\beta_{2}=1$ that
\begin{equation}
\begin{aligned}
&\sup_{t\in[1,T]}e^{\alpha_{1} t}\mathcal{E}(t)+\int_{1}^T
e^{\alpha_{1} t}\left(\mu\|\nabla u_{t}\|_{L^{2}}^{2}+ \kappa \|\sqrt{n_{f}}u_{t}\|_{L^{2}}^{2}\right){\rm{d}}t\\
&\leq C\exp{\left(\int_{0}^{1}\mu^{-3}\|\nabla u\|_{L^{2}}^{4}{\rm{d}}t\right)} \left( \mathcal{E}(1)+\int_{1}^{T}e^{\alpha_{1}t}(\alpha_{1}\mathcal{E}(t)+F(t)){\rm{d}}t\right)\\
&\leq C\Big(\mu^{-5}(\mu+\kappa)E_{0}M +E_{0}+1+\mu^{-1}\kappa +\mu^{-5}E_{0}\Big)E_{0}.
\end{aligned}
\end{equation}
This completes the proof of Lemma \ref{lem5}.
\end{proof}

With the help of Lemmas \ref{lem3}--\ref{lem5}, we are ready to prove Proposition \ref{prop3}.

\begin{proof}[Proof of Proposition {\rm\ref{prop3}}]
Assume that the initial data $(\rho_{0},u_{0},f_{0})$ satisfies \eqref{indacd}. For any given time $T>0$, let $(\rho,u,f)$ be the smooth solution to the Cauchy problem \eqref{NSV}--\eqref{inda} satisfying \eqref{priori}. Our goal is to justify the a-priori estimate \eqref{clspri}. First, by Lemmas \ref{lem3} and \ref{lem4}, there exists a constant $\widetilde{C}_{1}=\widetilde{C}_{1}(\|\rho_{0}\|_{L^{\infty}},\|f_{0}\|_{L^{1}_{v}(L^{\infty}_{x})})$ such that
\begin{equation}\label{uniformL4}
\begin{aligned}
\int_{0}^T\mu^{-3}\|\nabla u\|_{L^{2}}^{4}{\rm{d}}t
&\leq \mu^{-5}\sup_{t\in(0,T]}\left(\mu\|\nabla u\|_{L^{2}}^{2}\right)\int_{0}^{T}\mu\|\nabla u\|_{L^{2}}^{2}{\rm{d}}t \\
&\leq \widetilde{C}_{1}M(\mu^{-4} +\mu^{-5}\kappa)E_{0}\\&<\frac{1}{6},
\end{aligned}
\end{equation}
provided that
\begin{equation*}
E_{0}<\min\{\mu^{4},\,\mu^{5}\} \kappa^{-1}\varepsilon_{1}~~\text{with}~~  \varepsilon_{1}:=\frac{1}{12\widetilde{C}_{1}M}.
\end{equation*}

Next, we establish the $L^{1}(0,T;L^{\infty})$-norm of $u$. It follows from the Gagliardo-Nirenberg inequality that
\begin{equation*}
\int_{0}^T\kappa\|u\|_{L^{\infty}}{\rm{d}}t\leq C\kappa\int_{0}^{T}\|\nabla u\|_{L^{2}}^{\frac{1}{2}}\|\nabla^{2} u\|_{L^{2}}^{\frac{1}{2}}{\rm{d}}t.
\end{equation*}
For the short time $T\leq 1$, thanks to \eqref{dsin} with $\beta_{1}=1$, we get
\begin{equation*}
\begin{aligned}
\int_{0}^{T}\|\nabla u\|_{L^{2}}^{\frac{1}{2}}\|\nabla^{2} u\|_{L^{2}}^{\frac{1}{2}}{\rm{d}}t&\leq C\mu^{-\frac{3}{4}}\sup_{t\in(0,1]}\Big(t\mu\|\nabla u\|_{L^{2}}^{2}\Big)^\frac{1}{4}\left(\int_{0}^{1}t\mu^{2}\|\nabla^{2}u\|_{L^{2}}^{2}{\rm{d}}t\right)^\frac{1}{4}\left(\int_{0}^{1}t^{-\frac{2}{3}}{\rm{d}}t\right)^\frac{3}{4}\\
&\leq C\mu^{-\frac{3}{4}}E_{0}^\frac{1}{2}.
\end{aligned}
\end{equation*}
As for the large time $T\geq1$, the time-integral on $(0,1]$ can be handled similarly as above, and on the rest part $[1,T]$, one deduces from \eqref{exdsin} that
\begin{equation*}
\begin{aligned}
&\int_{1}^{T}\|\nabla u\|_{L^{2}}^{\frac{1}{2}}\|\nabla^{2} u\|_{L^{2}}^{\frac{1}{2}}{\rm{d}}t\\
&\leq C\mu^{-\frac{3}{4}}\sup_{t\in[1,T]}\Big(e^{\alpha_{1}t}\mu\|\nabla u\|_{L^{2}}^{2}\Big)^\frac{1}{4}\left(\int_{1}^Te^{\alpha_{1}t}\mu^{2}\|\nabla^{2}u\|_{L^{2}}^{2}{\rm{d}}t\right)^\frac{1}{4}\left(\int_{1}^Te^{-\frac{2}{3}\alpha_{1}t}{\rm{d}}t\right)^\frac{3}{4}\\
&\leq C\alpha_{1}^{-\frac{3}{4}}\mu^{-\frac{3}{4}}E_{0}^\frac{1}{2}\\
&\leq C(\mu^{-\frac{3}{4}}+\mu^{-\frac{3}{2}}) E_{0}^\frac{1}{2},
\end{aligned}
\end{equation*}
where we used the fact from \eqref{alpha1} that $\alpha_{1}^{-1}\leq C(1+\mu^{-1})$.  It thus holds for some constant $\widetilde{C}_{2}=\widetilde{C}_{2}(\|\rho_{0}\|_{L^{\frac{3}{2}}\cap L^{\infty}},\|f_{0}\|_{L^{1}_{x,v}\cap L^{1}_{v}(L^{\infty}_{x})}))$ that
\begin{equation}\label{uL1infty}
\int_{0}^T\kappa\|u\|_{L^{\infty}}{\rm{d}}t\leq \widetilde{C}_{2} \kappa (\mu^{-\frac{3}{4}}+\mu^{-\frac{3}{2}})E_{0}^\frac{1}{2}<\frac{1}{6},
\end{equation}
provided that
\begin{equation*}
E_{0}<\min\{\mu^\frac{3}{2}, \mu^{3}\} \kappa^{-2}\varepsilon_{2}~~\text{with}~~ \varepsilon_{2}:=\frac{1}{144\widetilde{C}_{2}^{2}}.
\end{equation*}

Finally, it suffices to control the $L^{1}(0,T;L^{\infty})$-norm of $\nabla u$. It can be derived from the Gagliardo-Nirenberg inequality \eqref{GNSine} that
\begin{equation}\label{nuif}
\|\nabla u\|_{L^{\infty}}\leq C\|\nabla u\|_{L^{2}}^{\frac{1}{4}}\|\nabla^{2} u\|_{L^{6}}^{\frac{3}{4}},
\end{equation}
which requires the estimate of $\|\nabla^{2} u\|_{L^{6}}$. To that matter, we employ \eqref{re2} to $\eqref{NSV}_{2}$, making use of the bounds \eqref{rhoinfty}--\eqref{jf}, and then get
\begin{equation*}
\begin{aligned}
&\mu\|\nabla^{2}u\|_{L^{6}}+\|\nabla P\|_{L^{6}}\\
&\leq C\Big\|\rho u_{t}+\rho u\cdot\nabla u+\int_{\mathbb{R}^{3}}\kappa(u-v)f{\rm{d}}v\Big\|_{L^{6}}\\
&\leq C\|\rho\|_{L^{\infty}} \|u_{t}\|_{L^{6}}+C\|\rho\|_{L^{\infty}}\|u\|_{L^{6}}\|\nabla u\|_{L^{\infty}}+C\kappa \Big\| \int_{\mathbb{R}^{3}}(u-v)f{\rm{d}}v\Big\|_{L^{2}}^{\frac{1}{3}} \Big\| \int_{\mathbb{R}^{3}}(u-v)f{\rm{d}}v\Big\|_{L^{\infty}}^{\frac{2}{3}}\\
&\leq C\|\rho_{0}\|_{L^{\infty}} \|\nabla u_{t}\|_{L^{2}}+C\|\rho_{0}\|_{L^{\infty}}\|\nabla u\|_{L^{2}}^{\frac{5}{4}} \|\nabla^{2} u\|_{L^{6}}^{\frac{3}{4}}\\
&\quad+C\kappa \|f_{0}\|_{L^{1}_{v}(L^{\infty}_{x})}^{\frac{1}{6}} \left(\int_{\mathbb{R}^{3}\times\mathbb{R}^{3}} |u-v|^{2}f{\rm{d}}v{\rm{d}}x\right)^{\frac{1}{6}} \left(\|u\|_{L^{\infty}} \|f_{0}\|_{L^{1}_{v}(L^{\infty}_{x})}+\| (1+|v|)f_{0}\|_{L^{1}_{v}(L^{\infty}_{x})}\right)^{\frac{2}{3}}\\
&\leq \frac{\mu}{2}\|\nabla^{2}u\|_{L^{6}}+C\|\nabla u_{t}\|_{L^{2}}+C\mu^{-3}\|\nabla u\|_{L^{2}}^{5}+C\kappa(1+\|u\|_{L^{\infty}}) \left(\int_{\mathbb{R}^{3}\times\mathbb{R}^{3}} |u-v|^{2}f{\rm{d}}v{\rm{d}}x\right)^{\frac{1}{6}},
\end{aligned}
\end{equation*}
from which we infer
\begin{equation}\label{dfgg}
\begin{aligned}
&\|\nabla^{2}u\|_{L^{6}}+\mu^{-1}\|\nabla P\|_{L^{6}}\\
&\leq C\mu^{-1}\|\nabla u_{t}\|_{L^{2}}+C\mu^{-\frac{13}{2}}D(t)^{\frac{5}{2}}+C\mu^{-1}\kappa^\frac{5}{6}(1+\|u\|_{L^{\infty}}) D(t)^{\frac{1}{6}}.
\end{aligned}
\end{equation}
Then, it follows from \eqref{nuif} and \eqref{dfgg} that
\begin{equation}\label{nablauinfty}
\begin{aligned}
\|\nabla u\|_{L^{\infty}}
&\leq C\mu^{-\frac{5}{4}}D(t)^{\frac{1}{8}}\left(\mu \|\nabla u_{t}\|_{L^{2}}^{2}\right)^\frac{3}{8}+C\mu^{-5}D(t)^{2}\\
&\quad+C\mu^{-\frac{7}{8}}\kappa^{-\frac{1}{8}}(1+\|u\|_{L^{\infty}}^\frac{3}{4})D(t)^{\frac{1}{4}}.
\end{aligned}
\end{equation}
By virtue of \eqref{nablauinfty}, it can be derived from \eqref{priori}, \eqref{eptd}, \eqref{dsin} and \eqref{twds}  that for $T\leq 1$,
\begin{equation*}
\begin{aligned}
\int_{0}^{T}\|\nabla u\|_{L^{\infty}}{\rm{d}}t&\leq C\mu^{-\frac{5}{4}}\sup_{t\in(0,1]}\Big(tD(t)\Big)^{\frac{1}{8}}\left(\int_{0}^{1}t\mu\|\nabla u_{t}\|_{L^{2}}^{2} {\rm{d}}t\right)^{\frac{3}{8}}\left(\int_{0}^{1}t^{-\frac{4}{5}}{\rm{d}}t\right)^{\frac{5}{8}}\\
&\quad+C\mu^{-5}\sup_{t\in(0,T]} D(t)\int_{0}^{T}D(t){\rm{d}}t+C\mu^{-\frac{7}{8}}\kappa^{\frac{5}{8}}\left(\int_{0}^{1}D(t){\rm{d}}t\right)^\frac{1}{4} \\
&\quad+C\mu^{-\frac{7}{8}}\kappa^{-\frac{1}{8}}\left(\int_{0}^{1} \kappa \|u\|_{L^{\infty}}{\rm{d}}t\right)^{\frac{3}{4}} \left(\int_{0}^{1} D(t){\rm{d}}t\right)^{\frac{1}{4}} \\
&\leq C \mu^{-\frac{5}{4}}(\mu+\kappa)^{\frac{3}{8}}\Big(\mu^{-5}(\mu+\kappa)E_{0} M+\mu^{-5}E_{0} +\mu^{-1}\kappa+E_{0}+1 \Big)^\frac{3}{8}E_{0}^{\frac{1}{8}}M^\frac{3}{8}\\
&\quad+C\mu^{-5}(\mu+\kappa)E_{0}M+C\mu^{-\frac{7}{8}}\kappa^\frac{5}{8}E_{0}^\frac{1}{4}.
\end{aligned}
\end{equation*}
On the other hand, due to the large-time estimates \eqref{exdsin} and \eqref{extwds}, it also holds for $T\geq 1$ that
\begin{equation*}
\begin{aligned}
\int_{1}^T \|\nabla u\|_{L^{\infty}}{\rm{d}}t&\leq C\mu^{-\frac{5}{4}}\sup_{t\in[1,T]}\left(e^{\alpha_{1}t}D(t)\right)^{\frac{1}{8}}\left(\int_{1}^Te^{\alpha_{1}t}\mu\|\nabla u_{t}\|_{L^{2}}^{2} {\rm{d}}t\right)^{\frac{3}{8}}\left(\int_{1}^Te^{-\frac{4}{5}\alpha_{1}t}{\rm{d}}t\right)^{\frac{5}{8}}\\
&\quad+C\mu^{-5}\sup_{t\in[1,T]} D(t)\int_{1}^{T}e^{\alpha_{1}t} D(t){\rm{d}}t\\
&\quad+C\mu^{-\frac{7}{8}}\kappa^{\frac{5}{8}}\left(\int_{1}^{T}e^{\alpha_{1} t}D(t){\rm{d}}t\right)^{\frac{1}{4}}\left(\int_{1}^{T}e^{-\frac{1}{3}\alpha_{1}t}{\rm{d}}t\right)^{\frac{3}{4}}\\
&\quad+C\mu^{-\frac{7}{8}}\kappa^{-\frac{1}{8}}\left(\int_{1}^{T} \kappa \|u\|_{L^{\infty}}{\rm{d}}t\right)^{\frac{3}{4}} \left(\int_{1}^T D(t){\rm{d}}t\right)^{\frac{1}{4}} \\
&\leq C\mu^{-\frac{5}{4}}(1+\mu^{-1})^{\frac{5}{8}}\Big(\mu^{-5}(\mu+\kappa)E_{0} M+\mu^{-5}E_{0} +\mu^{-1}\kappa+E_{0}+1 \Big)^\frac{3}{8}E_{0}^{\frac{1}{2}}\\
&\quad+C\mu^{-5}(\mu+\kappa)E_{0}M+C\mu^{-\frac{7}{8}}(1+\mu^{-1})^\frac{3}{4}\kappa^\frac{5}{8}E_{0}^\frac{1}{4}.
\end{aligned}
\end{equation*}
To sum up, noticing that $\kappa\geq 1$, it holds that
\begin{equation}\label{nablauinfty11}
\begin{aligned}
&\int_{0}^{T}20\kappa\|\nabla u\|_{L^{\infty}}{\rm{d}}t\\
&\leq \widetilde{C}_{3}\kappa^{2}\Big\{(\mu^{-\frac{7}{8}}+\mu^{-\frac{13}{8}})E_{0}^\frac{1}{8}+\mu^{-\frac{7}{8}}E_{0}^\frac{1}{4}+(\mu^{-\frac{7}{8}}+\mu^{-\frac{25}{8}})E_{0}^\frac{1}{2}\\
&\qquad\qquad+(\mu^{-\frac{5}{4}}+\mu^{-\frac{15}{4}})E_{0}^\frac{7}{8}+(\mu^{-5}+\mu^{-4})E_{0}\Big\}<\frac{1}{6},
\end{aligned}
\end{equation}
for some constant $\widetilde{C}_{3}=\widetilde{C}_{3}(M, \|\rho_{0}\|_{L^\frac{3}{2}\cap L^{\infty}}, \|f_{0}\|_{L^{1}_{x,v}}, \|(1+|v|^{2}) f_{0}\|_{L_{v}^{1}(L_{x}^{\infty})})$. Thus, we let the following five quantities:
\begin{equation*}
\begin{aligned}
&\Big(\mu^{-\frac{7}{8}}+\mu^{-\frac{13}{8}}\Big)E_{0}^\frac{1}{8},\quad\mu^{-\frac{7}{8}}E_{0}^\frac{1}{4},\quad\Big(\mu^{-\frac{7}{8}}+\mu^{-\frac{25}{8}}\Big)E_{0}^\frac{1}{2},\\
&\Big(\mu^{-\frac{5}{4}}+\mu^{-\frac{15}{4}}\Big)E_{0}^\frac{7}{8},\quad\Big(\mu^{-5}+\mu^{-4}\Big)E_{0}
\end{aligned}
\end{equation*}
be less than $1/(30\widetilde{C}_{3}\kappa^{2})$, i.e., there exists a small constant $\varepsilon_{3}>0$ only depending on $\widetilde{C}_{3}$ such that
\begin{equation*}
E_{0}<\min\{\mu^\frac{10}{7},\,\mu^{13}\}\kappa^{-16}\varepsilon_{3}.
\end{equation*}

Therefore, combining the above estimates \eqref{uniformL4},  \eqref{uL1infty} and \eqref{nablauinfty11} and taking $E_{0}$ satisfying \eqref{inenrgcd} with $\varepsilon_{0}:=\min\{\varepsilon_{1},\varepsilon_{2},\varepsilon_{3}\}$, we end up with \eqref{clspri} and complete the proof of Proposition \ref{prop3}.
\end{proof}

\subsection{Large-time behavior of the distribution function}

We further capture the concentration phenomena of the distribution function $f$.

\begin{lemma}\label{lem6}
Let $T>0$ and $(\rho,u,f)$ be the smooth solution to \eqref{NSV} for $t\in(0,T]$ with the initial data $(\rho_{0},u_{0},f_{0})$ satisfying \eqref{indacd}. Then under the condition \eqref{priori}, it holds for $\alpha_{2}=\min\left\{\frac{\kappa}{2},\frac{\alpha_{1}}{4}\right\}$ that
\begin{equation}\label{Sp1}
\left\{
\begin{aligned}
&\|j_{f}(t)\|_{L^{\infty}}\leq C_{5} e^{-\alpha_{2} t}\left(\||v|{f_{0}}\|_{L^{1}_{v}(L^{\infty}_{x})}+E_{0}^{\frac{1}{2}}\right),\\
&\|e_{f}(t)\|_{L^{\infty}}\leq C_{5}e^{-\alpha_{2} t}\Big(\||v|^{2}f_{0}\|_{L^{1}_{v}(L^{\infty}_{x})}+ E_{0}^\frac{1}{2}\Big),\\
&{\rm{Supp}}_{v}f(t,x,v)\subset \Big\{v\in\mathbb{R}^{3}~|~|v|\leq C_{5} e^{-\alpha_{2} t}\left(R_{0}+E_{0}^{\frac{1}{2}}\right)\Big\},
\end{aligned}
\right.
\end{equation}
where $(t,x)\in(0,T]\times\mathbb{R}^{3}$, and as $T\rightarrow +\infty$, there exists a profile $n^{*}$ such that
\begin{equation}\label{sptf}
W_{1}(f(t),n^{*} \otimes\delta_{v})+W_{1}(n_{f}(t),n^{*})\leq C_{5} e^{-\alpha_{1} t}E^{\frac{1}{2}}_{0}.
\end{equation}
Here $E_{0}$ and $\alpha_{1}$ are given by \eqref{inenrg} and \eqref{alpha1}, respectively, and $C_{5}>0$ is a constant depending on $\mu$, $\kappa$, $M$, $\|\rho_{0}\|_{L^{\frac{3}{2}}\cap L^{\infty}}$ and $\|f_{0}\|_{L^{1}_{x,v}\cap L^{1}_{v}(L^{\infty}_{x})}$.
\end{lemma}

\begin{proof}
In terms of \eqref{jflinfty}, \eqref{dsin} and \eqref{exdsin}, it holds for $\alpha_{2}=\min\left\{\frac{\kappa}{2},\frac{\alpha_{1}}{4}\right\}$ that
\begin{equation*}
\begin{aligned}
\|j_{f}\|_{L^{\infty}}
&\leq 2e^{-\kappa t}\||v|{f_{0}}\|_{L^{1}_{v}(L^{\infty}_{x})}\\
&\quad+2\|f_{0}\|_{L^{1}_{x}(L^{\infty}_{x})}\left\{\int_{0}^{1}+\int_{1}^{t}\right\} e^{-\kappa (t-\tau)}\kappa\|\nabla u\|_{L^{2}}^{\frac{1}{2}}\|\nabla^{2} u\|_{L^{2}}^{\frac{1}{2}}{\rm{d}}\tau\\
&\leq 2e^{-\kappa t}\||v|{f_{0}}\|_{L^{1}_{v}(L^{\infty}_{x})}\\
&\quad+Ce^{-\kappa t}\kappa\mu^{-\frac{3}{4}}\sup_{\tau\in(0,1]}\left(\tau\mu\|\nabla u\|_{L^{2}}^{2}\right)^\frac{1}{4}\left(\int_{0}^{1}\tau\mu^{2}\|\nabla^{2} u\|_{L^{2}}^{2}{\rm{d}}\tau\right)^{\frac{1}{4}}\\
&\quad +C\kappa\mu^{-\frac{3}{4}} \sup_{\tau\in[1,t]}\left(e^{\alpha_{1}\tau}\mu\|\nabla u\|_{L^{2}}^{2}\right)^\frac{1}{4}\left(\int_{1}^{t}e^{\alpha_{1}\tau}\mu^{2}\|\nabla^{2} u\|_{L^{2}}^{2}{\rm{d}}\tau\right)^{\frac{1}{4}}\left(\int_{1}^te^{-\frac{4}{3}\kappa(t-\tau)}e^{-\frac{2}{3}\alpha_{1}\tau}{\rm{d}}\tau\right)^\frac{3}{4}\\
&\leq Ce^{-\alpha_{2} t}\left(\||v|{f_{0}}\|_{L^{1}_{v}(L^{\infty}_{x})}+\kappa\mu^{-\frac{3}{4}}E_{0}^{\frac{1}{2}}\right).
\end{aligned}
\end{equation*}
Similarly, based on \eqref{eflinfty}, we deduce that
\begin{equation}
\begin{aligned}
\|e_{f}\|_{L^{\infty}}
&\leq 2e^{-2\kappa t}\||v|^{2}f_{0}\|_{L^{1}_{v}(L^{\infty}_{x})}\\
&\quad+2\|f_{0}\|_{L^{1}_{x}(L^{\infty}_{x})}\left\{\int_{0}^{1}+\int_{1}^{t}\right\} e^{-\kappa (t-\tau)}\kappa^{2}\|\nabla u\|_{L^{2}}\|\nabla^{2} u\|_{L^{2}}{\rm{d}}\tau\\
&\leq 2e^{-2\kappa t}\||v|^{2}f_{0}\|_{L^{1}_{v}(L^{\infty}_{x})}\\
&\quad+Ce^{-\kappa t}\kappa^{2}\mu^{-\frac{3}{2}}\left(\int_{0}^{1} \mu\|\nabla u\|_{L^{2}}^{2}{\rm{d}}\tau\right)^\frac{1}{2}\left(\int_{0}^{1}\mu^{2}\|\nabla^{2} u\|_{L^{2}}^{2}{\rm{d}}\tau\right)^\frac{1}{2}\\
&\quad +C\kappa^{2}\mu^{-\frac{3}{2}}\sup_{\tau\in[1,t]}\left(\mu\|\nabla u\|_{L^{2}}^{2}\right)^\frac{1}{2}\left(\int_{1}^{t}e^{\alpha_{1}\tau}\mu^{2}\|\nabla^{2} u\|_{L^{2}}^{2}{\rm{d}}\tau\right)^{\frac{1}{2}}\left(\int_{1}^te^{-2\kappa(t-\tau)}e^{-\alpha_{1}\tau}{\rm{d}}\tau\right)^\frac{1}{2}\\
&\leq Ce^{-\alpha_{2} t}\left(\||v|^{2}f_{0}\|_{L^{1}_{v}(L^{\infty}_{x})}+\kappa^{2}\mu^{-\frac{3}{2}}M^\frac{1}{2}E_{0}^\frac{1}{2}\right).
\end{aligned}
\end{equation}

Then, we deal with the large-time behavior of the compact support of $f$ with respect to $v$. Due to the characteristics \eqref{chrcrv}, it is straightforward to obtain
\begin{equation}\label{repv}
v=e^{-\kappa t}V(0;t,x,v)+\int_{0}^te^{-\kappa(t-\tau)} \kappa u(\tau,X(\tau;t,x,v)){\rm{d}}\tau.
\end{equation}
Recalling \eqref{dstrbt} that
\begin{equation*}
f(t,x,v)=e^{3\kappa t}f_{0}(X(0;t,x,v),V(0;t,x,v)),
\end{equation*}
we deduce
\begin{equation*}
\begin{aligned}
f(t,x,v)\neq 0
&\Rightarrow V(0;t,x,v)\in{\rm{Supp}}_{v}f_{0}(X(0;t,x,v),V(0;t,x,v))\\
&\Rightarrow |V(0;t,x,v)|\leq R_{0}.
\end{aligned}
\end{equation*}
Hence, in a similar argument as for estimating $\|j_{f}\|_{L^{\infty}}$, one has
\begin{equation*}
|v|\leq e^{-\kappa t}R_{0}+\int_{0}^te^{-\kappa(t-\tau)}\kappa\|u\|_{L^{\infty}}{\rm{d}}\tau\leq C e^{-\alpha_{2} t}\left(R_{0}+\kappa\mu^{-\frac{3}{4}}E_{0}^{\frac{1}{2}}\right),
\end{equation*}
for any $(t,x)\in(0,T]\times\mathbb{R}^{3}$ and $v\in{\rm{Supp}}_{v}f(t,x,v)$.

Furthermore, we establish the large-time behavior of the distribution function $f$ and the moment $n_{f}$. By the Monge-Kantorovitch duality (refer to Lemma \ref{MKdual}), we have
\begin{equation*}
\begin{aligned}
&W_{1}(f(t),n_{f}(t)\otimes\delta_{v})\\
&=\sup_{\|\nabla_{x,v}\psi\|_{L^{\infty}}\leq1}\int_{\mathbb{R}^{3}\times\mathbb{R}^{3}} f(t,x,v)\left(\psi(x,v)-\psi(x,0)\right){\rm{d}}v{\rm{d}}x\\
&\leq\sup_{\|\nabla_{x,v}\psi\|_{L^{\infty}}\leq1}\|\nabla_{v}\psi\|_{L^{\infty}}\int_{\mathbb{R}^{3}\times\mathbb{R}^{3}} |v|f{\rm{d}}v{\rm{d}}x\\
&\leq\left(\int_{\mathbb{R}^{3}\times\mathbb{R}^{3}} f {\rm{d}}v{\rm{d}}x\right)^{\frac{1}{2}}\left(\int_{\mathbb{R}^{3}\times\mathbb{R}^{3}}|v|^{2}f {\rm{d}}v{\rm{d}}x\right)^{\frac{1}{2}}\\
&\leq\sqrt{2}\|n_{f}\|_{L^{1}}^{\frac{1}{2}}E(t)^{\frac{1}{2}},
\end{aligned}
\end{equation*}
which together with \eqref{nf1} and \eqref{eptd} leads to
\begin{equation}\label{W11f}
W_{1}(f(t),n_{f}(t)\otimes\delta_{v})\leq Ce^{-\frac{\alpha_{1}}{2}t}E_{0}^{\frac{1}{2}}.
\end{equation}
In addition, the integration of $\eqref{NSV}_{4}$ with respect to $v\in\mathbb{R}^{3}$ yields
\begin{equation*}
(n_{f})_{t}+{\rm{div}}_{x}j_{f}=0.
\end{equation*}
Thus it holds for any smooth function $\varphi(x)\in C^{\infty}(\mathbb{R}^{3})$ that
\begin{equation*}
\frac{{\rm{d}}}{{\rm{d}}t}\int_{\mathbb{R}^{3}}\varphi n_{f}{\rm{d}}x-\int_{\mathbb{R}^{3}}\nabla_{x}\varphi j_{f}{\rm{d}}x=0,
\end{equation*}
which implies for any $t\in(0,T]$ that
\begin{equation*}
\begin{aligned}
\left|\int_{\mathbb{R}^{3}}\varphi n_{f}(T){\rm{d}}x-\int_{\mathbb{R}^{3}}\varphi n_{f}(t){\rm{d}}x\right|&\leq\int_{t}^T\left(\int_{\mathbb{R}^{3}\times\mathbb{R}^{3}}|v|^{2}f {\rm{d}}v{\rm{d}}x\right)^{\frac{1}{2}}\left(\int_{\mathbb{R}^{3}} n_{f}|\nabla_{x}\varphi|^{2}{\rm{d}}x\right)^{\frac{1}{2}}{\rm{d}}\tau\\
&\leq\|\nabla_{x}\varphi\|_{L^{\infty}}\int_{t}^T\|n_{f}\|_{L^{1}}^{\frac{1}{2}}E(\tau)^{\frac{1}{2}}{\rm{d}}\tau.
\end{aligned}
\end{equation*}
More precisely, if $\|\nabla\varphi\|_{L^{\infty}}\leq 1$, in terms of \eqref{nf1} and \eqref{eptd}, we have
\begin{equation}\label{mmm}
\left|\int_{\mathbb{R}^{3}}\varphi n_{f}(T){\rm{d}}x-\int_{\mathbb{R}^{3}}\varphi n_{f}(t){\rm{d}}x\right|\leq \|{f_{0}}\|_{L^{1}_{x,v}}^{\frac{1}{2}}E_{0}^{\frac{1}{2}}\int_{t}^T e^{-\frac{\alpha_{1}}{2}\tau}{\rm{d}}\tau,
\end{equation}
and this can be extended to Lipschitz functions $\varphi$ by a standard approximation argument so that the Monge-Kantorovitch duality formula in Lemma \ref{MKdual} leads to
\begin{equation*}
W_{1}(n_{f}(t),n_{f}(T))\leq CE_{0}^{\frac{1}{2}}\alpha_{1}^{-1}(e^{-\frac{\alpha_{1}}{2}t}-e^{-\frac{\alpha_{1}}{2}T}).
\end{equation*}
Therefore the Cauchy criterion for $n_{f}(t)$ is verified as $T\rightarrow \infty$, which means there exist some measure $n^{*}$ such that $n_{f}(t)\rightarrow n^{*}$ as $t\rightarrow \infty$ in the sense of $1$-Wasserstein distance. By letting $T\rightarrow \infty$ in \eqref{mmm}, we conclude
\begin{equation*}
W_{1}(n_{f}(t),n^{*})\leq Ce^{-\frac{\alpha_{1}}{2}t}E_{0}^{\frac{1}{2}},
\end{equation*}
which together with \eqref{W11f} leads to
\begin{equation*}
W_{1}(f(t),n^{*}\otimes\delta_{v})\leq Ce^{-\frac{\alpha_{1}}{2}t}E_{0}^{\frac{1}{2}}.
\end{equation*}
Thus, we complete the proof of Lemma \ref{lem6}.
\end{proof}

\subsection{Higher order estimates}

We establish some regularity estimates of the solution $(\rho,u,f)$, which provides the higher order decay estimates of $u$ and are useful for the further proof on the uniqueness of the strong solution.

\begin{lemma}\label{lem7}
Let $T>1$ and $(\rho,u,f)$ be the smooth solution to \eqref{NSV} for $t\in(0,T]$ with the initial data $(\rho_{0},u_{0},f_{0})$ satisfying \eqref{indacd}. Then under the condition \eqref{priori}, it holds
\begin{equation}\label{PL2}
\int_{0}^{T}\|P(t)\|_{L^{2}}^{2}{\rm{d}}t\leq C_{6}.
\end{equation}
Furthermore, for any $\beta_{3}\in[0,1]$, we have the short-time estimates
\begin{equation}\label{twdsjj}
\begin{aligned}
&\sup_{t\in(0,1]}t^{1+\beta_{3}}\left(\|\nabla u(t)\|_{H^{1}}^{2}+\|P(t)\|_{H^{1}}^{2}\right)\\
&+\int_{0}^{1} t^{1+\beta_{3}}\Big(\|u_{t}(t)\|_{L^{6}}^{2}+\|\nabla^{2} u(t)\|_{L^{6}}^{2}+\|\nabla P(t)\|_{L^{6}}^{2} \Big){\rm{d}}t\\& \leq C_{6},
\end{aligned}
\end{equation}
and the large-time exponential stability
\begin{equation}\label{extwdsjj}
\begin{aligned}
&\sup_{t\in[1,T]}e^{\alpha_{1} t}\left(\|\nabla u(t)\|_{H^{1}}^{2}+\|P(t)\|_{H^{1}}^{2}\right)\\
&+\int_{1}^{T} e^{\frac{\alpha_{1}}{4}t}\Big(\|u_{t}(t)\|_{L^{6}}^{2}+\|\nabla^{2} u(t)\|_{L^{6}}^{2}+\|\nabla P(t)\|_{L^{6}}^{2} \Big){\rm{d}}t \\&\leq C_{6},
\end{aligned}
\end{equation}
where $C_{6}>0$ is a constant only depending on $\mu$, $\kappa$, $M$, $\|\rho_{0}\|_{L^{\frac{3}{2}}\cap L^{\infty}}$, $\|f_{0}\|_{L^{1}_{x,v}}$ and $\|(1+|v|^{2})f_{0}\|_{ L^{1}_{v}(L^{\infty}_{x})}$.
\end{lemma}

\begin{proof}
In view of \eqref{re1} in Lemma \ref{SSineq} to $\eqref{NSV}_{2}$, it follows from \eqref{esro}, \eqref{rhoinfty}, \eqref{nf} and the Gagliardo-Nirenberg inequality \eqref{GNSine} that
\begin{equation}\label{mmmd}
\begin{aligned}
\mu^{2}\|\nabla u\|_{L^{2}}^{2}+\|P\|_{L^{2}}^{2}&\leq C\Big\|\rho u_{t}+\rho u\cdot\nabla u+\int_{\mathbb{R}^{3}}\kappa (u-v)f{\rm{d}}v\Big\|_{L^{\frac{6}{5}}}^{2}\\
&\leq C\|\sqrt{\rho}\|_{L^{3}}^{2}\|\sqrt{\rho}u_{t}\|_{L^{2}}^{2}+C\|\rho\|_{L^{\infty}}\|\rho\|_{L^{\frac{3}{2}}}\|u\|_{L^{6}}^{2}\|\nabla u\|_{L^{3}}^{2}\\
&\quad+\kappa^{2}\|n_{f}\|_{L^\frac{3}{2}} \int_{\mathbb{R}^{3}\times\mathbb{R}^{3}}|u-v|^{2}f{\rm{d}}v{\rm{d}}x\\
&\leq C\|\rho_{0}\|_{L^{\frac{3}{2}}}\|\sqrt{\rho} u_{t}\|_{L^{2}}^{2}+C\|\rho_{0}\|_{L^{\infty}}\|\rho_{0}\|_{L^\frac{3}{2}} \|\nabla u\|_{L^{2}}^{3}\|\nabla^{2} u\|_{L^{2}}\\
&\quad+C\|f_{0}\|_{L_{x,v}^{1}}^\frac{2}{3}\|f_{0}\|_{L^{1}_{v}(L^{\infty}_{x})}^\frac{1}{3} \kappa^{2}\int_{\mathbb{R}^{3}\times\mathbb{R}^{3}}|u-v|^{2}f{\rm{d}}v{\rm{d}}x \\
&\leq C\mathcal{E}(t)+\frac{1}{2}\mu^{2}\|\nabla^{2} u\|_{L^{2}}^{2}+C\mu^{-3}\|\nabla u\|_{L^{2}}^{4} D(t),
\end{aligned}
\end{equation}
which, together with \eqref{priori} and \eqref{dsin0}, leads to
\begin{equation*}
\begin{aligned}
\int_{0}^T\|P\|_{L^{2}}^{2}{\rm{d}}t
&\leq C\int_{0}^T(\mathcal{E}(t)+\mu^{2}\|\nabla^{2} u\|_{L^{2}}^{2}){\rm{d}}t+C\sup_{t\in(0,T]} D(t) \int_{0}^T\mu^{-3}\|\nabla u\|_{L^{2}}^{4}{\rm{d}}t\\&\leq C_{6}.
\end{aligned}
\end{equation*}
Moreover, by \eqref{mmmmm3} and \eqref{mmmd} one has
\begin{equation}\label{mmmd1}
\begin{aligned}
\mu^{2}\|\nabla u\|_{H^{1}}^{2}+\|P\|_{H^{1}}^{2}\leq C\mathcal{E}(t)+C\mu^{-5}D(t)^{3}.
\end{aligned}
\end{equation}
It can be derived from \eqref{dsin}, \eqref{twds} and \eqref{mmmd1} that
\begin{equation}\label{shorthigher}
\begin{aligned}
&\sup_{t\in(0,1]}t^{1+\beta_{3}}\Big(\mu^{2}\|\nabla u\|_{H^{1}}^{2}+\|P\|_{H^{1}}^{2}\Big)\\
&\leq C\sup_{t\in(0,1]}t^{1+\beta_{3}}\mathcal{E}(t)+C\mu^{-5}\sup_{t\in(0,1]}\left(t^\frac{{1+\beta_{3}}}{3}D(t)\right)^{3}\\&\leq C_{6}.
\end{aligned}
\end{equation}
Due to the Sobolev inequality, the estimate \eqref{twds} with $\beta_{2}=\beta_{3}$ directly gives
\begin{equation}\label{shorthigher1}
\begin{aligned}
\int_{0}^{1} t^{1+\beta_{3}}\mu\|u_{t}\|_{L^{6}}^{2}{\rm{d}}t\leq C\int_{0}^{1} t^{1+\beta_{3}}\mu\|\nabla u_{t}\|_{L^{2}}^{2}{\rm{d}}t\leq C_{6}.
\end{aligned}
\end{equation}
To deal with weighted $L^{2}(0,T;L^{6})$-norms of $\nabla^{2} u$ and $\nabla P$, we recall \eqref{dfgg} that
\begin{equation*}
\begin{aligned}
&\mu^{2}\|\nabla^{2}u\|_{L^{6}}^{2}+\|\nabla P\|_{L^{6}}^{2}\\
&\leq C\|\nabla u_{t}\|_{L^{2}}^{2}+C\mu^{-11}D(t)^{5}+C\kappa^\frac{5}{3}(1+\|u\|_{L^{\infty}}^{2}) D(t)^{\frac{1}{3}}.
\end{aligned}
\end{equation*}
It thus holds by \eqref{dsin0}, \eqref{shorthigher} and \eqref{shorthigher1} that
\begin{equation}\label{shorthigher2}
\begin{aligned}
&\int_{0}^{1}t^{1+\beta_{3}}\Big(\mu^{2}\|\nabla^{2}u\|_{L^{6}}^{2}+\|\nabla P\|_{L^{2}}\Big){\rm{d}}t\\
&\leq C\mu^{-1}\int_{0}^{1}t^{1+\beta_{3}}\mu\|\nabla u_{t}\|_{L^{2}}^{2}{\rm{d}}t+C\mu^{-11}\sup_{t\in(0,1]}  \Big(t^\frac{1+\beta_{3}}{5}D(t)\Big)^{5}\\
&\quad+C\kappa^\frac{5}{3} \left(1+\mu^{-2}\sup_{t\in(0,1]} \Big(t^{1+\beta_{3}}\mu^{2}\|\nabla u\|_{H^{1}}^{2}\Big)\right)\sup_{t\in(0,1]}D(t)^{\frac{1}{3}}\\
&\leq C_{6}.
\end{aligned}
\end{equation}
Hence, we derive the short-time estimate \eqref{twdsjj} from the combination of \eqref{shorthigher}--\eqref{shorthigher2}.

Then, we turn to the large-time exponential stability \eqref{extwdsjj}. Making use of \eqref{exdsin}, \eqref{extwds} and \eqref{mmmd1}, we have
\begin{equation}\label{largehigher}
\begin{aligned}
&\sup_{t\in[1,T]}e^{\alpha_{1}t}\Big(\mu^{2}\|\nabla u\|_{H^{1}}^{2}+\|P\|_{H^{1}}^{2}\Big)\\
&\leq C\sup_{t\in[1,T]}e^{\alpha_{1}t}\mathcal{E}(t)+C \mu^{-5}\sup_{t\in[1,T]}\left(e^{\alpha_{1}t}D(t)\right)^{3}\\&\leq C_{6}.
\end{aligned}
\end{equation}
One deduces from \eqref{extwds} that
\begin{equation}\label{largehigher1}
\begin{aligned}
\int_{1}^{T} e^{\alpha_{1}t}\mu\|u_{t}\|_{L^{6}}^{2}{\rm{d}}t\leq C\int_{0}^{T} e^{\alpha_{1}t}\mu\|\nabla u_{t}\|_{L^{2}}^{2}{\rm{d}}t\leq C_{6},
\end{aligned}
\end{equation}
and owing to \eqref{dsin}, \eqref{extwds} and \eqref{largehigher}, we also obtain
\begin{equation}\label{largehigher2}
\begin{aligned}
&\int_{1}^{T}e^{\frac{1}{4}\alpha_{1}t}\Big(\mu^{2}\|\nabla^{2}u\|_{L^{6}}^{2}+\|\nabla P\|_{L^{6}}\Big){\rm{d}}t\\
& \leq C\int_{1}^{T}e^{\alpha_{1}t}\|\nabla u_{t}\|_{L^{2}}^{2}{\rm{d}}t+C\mu^{-11}\sup_{t\in[1,T]}D(t)^{4}\int_{1}^{T} e^{\alpha_{1}t} D(t){\rm{d}}t\\
&\quad+C\kappa^\frac{5}{3}\left(1+\mu^{-2}\sup_{t\in[1,T]}\Big(e^{\alpha_{1}t}\mu^{2} \|\nabla u\|_{H^{1}}^{2}\Big)\right)\left(\int_{1}^{T} e^{\alpha_{1} t} D(t){\rm{d}}t\right)^{\frac{1}{3}} \left(\int_{1}^{T} e^{-\frac{\alpha_{1}}{8}t}{\rm{d}}t\right)^{\frac{2}{3}}\\&\leq C_{6}.
\end{aligned}
\end{equation}
Thus, we conclude from \eqref{largehigher}--\eqref{largehigher2} the estimate \eqref{extwdsjj} and complete the proof of Lemma \ref{lem7}.
\end{proof}

\begin{lemma}\label{lem8}
Let $T>0$ and $(\rho,u,f)$ be the smooth solution to \eqref{NSV} for $t\in(0,T]$ with the initial data $(\rho_{0},u_{0},f_{0})$ satisfying \eqref{indacd}. Then under the condition \eqref{priori}, it holds
\begin{equation}\label{nrhol2}
\sup_{t\in(0,T]}\|\nabla\rho(t)\|_{L^{2}}\leq 2\|\nabla\rho_{0}\|_{L^{2}}.
\end{equation}
\end{lemma}

\begin{proof}
Applying $\nabla$ to $\eqref{NSV}_{1}$ and noticing the incompressible condition $\eqref{NSV}_{3}$, we have
\begin{equation*}
\nabla\rho_{t}+u\cdot\nabla(\nabla\rho)+\nabla u\cdot\nabla\rho=0,
\end{equation*}
which gives rise to
\begin{equation}\label{nablarhoL2}
\frac{1}{2}\frac{{\rm{d}}}{{\rm{d}}t}\|\nabla\rho\|_{L^{2}}^{2}\leq\frac{3}{2}\|\nabla u\|_{L^{\infty}}\|\nabla\rho\|_{L^{2}}^{2}.
\end{equation}
Dividing \eqref{nablarhoL2} by $\sqrt{\|\nabla\rho\|_{L^{2}}^{2}+\eta_{0}}$ with $\eta_{0}>0$, integrating the resulting inequality over $[0,t]$ for $t\in(0,T]$, and then taking the limit $\eta_{0}\rightarrow0$, we obtain
\begin{equation*}
\begin{aligned}
\|\nabla\rho\|_{L^{2}}&\leq \|\nabla\rho_{0}\|_{L^{2}}\exp{\left(\frac{3}{2}\int_{0}^T\|\nabla u\|_{L^{\infty}}{\rm{d}}t\right)}\\
&\leq 2\|\nabla\rho_{0}\|_{L^{2}},
\end{aligned}
\end{equation*}
where one has used \eqref{clspri}. Thus, the proof of Lemma \ref{lem8} is completed.
\end{proof}

\begin{lemma}\label{lem9}
Let $T>0$ and $(\rho,u,f)$ be the smooth solution to \eqref{NSV} for $t\in(0,T]$ with the initial data $(\rho_{0},u_{0},f_{0})$ satisfying \eqref{indacd}. Then under the condition \eqref{priori}, it holds
\begin{equation}\label{nbfp}
\sup_{t\in(0,T]}\|f(t)\|_{W^{1,\frac{3}{2}}_{x,v}}\leq C_{T}\|f_{0}\|_{W^{1,\frac{3}{2}}_{x,v}}.
\end{equation}
\end{lemma}

\begin{proof}
Multiplying $\eqref{NSV}_{4}$ by $f^{\frac{1}{2}}$ and integrating the resulting equality by parts over $\mathbb{R}^{3}\times\mathbb{R}^{3}$, we get
\begin{equation*}
\frac{2}{3}\frac{{\rm{d}}}{{\rm{d}}t}\|f\|_{L^{\frac{3}{2}}_{x,v}}^{\frac{3}{2}}\leq \kappa\|f\|_{L^{\frac{3}{2}}_{x,v}}^{\frac{3}{2}},
\end{equation*}
which implies that
\begin{equation}\label{L32}
\sup_{t\in(0,T]}\|f(t)\|_{L^{\frac{3}{2}}_{x,v}}\leq e^{\kappa T}\|f_{0}\|_{L_{x,v}^{\frac{3}{2}}}.
\end{equation}
Moreover, it can be derived from $\eqref{NSV}_{4}$ that
\begin{equation}\label{nbxf}
(\partial_{x}^{i} f)_{t}+v\cdot\nabla_{x}\partial_{x}^{i}f+\kappa\partial_{x}^{i} u\cdot\nabla_{v}f+\kappa(u-v)\cdot\nabla_{v}\partial_{x}^{i} f-3\kappa \partial_{x}^{i}f=0,
\end{equation}
for $i=1,\,2,\,3$. Multiplying \eqref{nbxf} by $|\partial_{x}^{i}f|^{\frac{1}{2}}{\rm{sign}}\, \partial_{x}^{i} f$, summing up for $i=1,2,3$, and integrating the resulting equation over $\mathbb{R}^{3}\times\mathbb{R}^{3}$, we obtain
\begin{equation}\label{nxfp}
\begin{aligned}
\frac{2}{3}\frac{{\rm{d}}}{{\rm{d}}t}\|\nabla_{x}f\|_{L^{\frac{3}{2}}_{x,v}}^{\frac{3}{2}}
&\leq\kappa\|\nabla u\|_{L^{\infty}}\|\nabla_{v}f\|_{L^{\frac{3}{2}}_{x,v}}^{\frac{1}{2}}\|\nabla_{x} f\|_{L^{\frac{3}{2}}_{x,v}}+\kappa \|\nabla_{x}f\|_{L^{\frac{3}{2}}_{x,v}}^{\frac{3}{2}}\\
&\leq C(\kappa\|\nabla u\|_{L^{\infty}}+\kappa)\|(\nabla_{x}f,\nabla_{v}f)\|_{L^{\frac{3}{2}}_{x,v}}^{\frac{3}{2}}.
\end{aligned}
\end{equation}
Similarly, it yields
\begin{equation*}
\frac{2}{3}\frac{{\rm{d}}}{{\rm{d}}t}\|\nabla_{v}f\|_{L^{\frac{3}{2}}_{x,v}}^{\frac{3}{2}}\leq C\kappa\|(\nabla_{x}f,\nabla_{v}f)\|_{L^{\frac{3}{2}}_{x,v}}^{\frac{3}{2}},
\end{equation*}
which, together with \eqref{nxfp}, shows that
\begin{equation*}
\frac{2}{3}\frac{{\rm{d}}}{{\rm{d}}t}\|(\nabla_{x}f,\nabla_{v}f)\|_{L^{\frac{3}{2}}_{x,v}}^{\frac{3}{2}}\leq C(\kappa\|\nabla u\|_{L^{\infty}}+\kappa)\|(\nabla_{x}f,\nabla_{v}f)\|_{L^{\frac{3}{2}}_{x,v}}^{\frac{3}{2}},
\end{equation*}
and then, by the Gr\"onwall inequality, we have
\begin{equation}\label{nL32}
\begin{aligned}
&\sup_{t\in(0,T]}\|(\nabla_{x}f,\nabla_{v}f)(t)\|_{L^{\frac{3}{2}}_{x,v}}\\
&\leq\exp\left(C\int_{0}^T\kappa\|\nabla u\|_{L^{\infty}}{\rm{d}}t+C\kappa T\right)\|(\nabla_{x}f_{0},\nabla_{v}f_{0})\|_{L^{\frac{3}{2}}_{x,v}}.
\end{aligned}
\end{equation}
Hence, we conclude the estimate \eqref{nbfp} from the combination of \eqref{clspri}, \eqref{L32} and \eqref{nL32}, and the proof of Lemma \ref{lem9} is completed.
\end{proof}

\section{Proofs of main results}\label{secprf}

With the help of all the estimates obtained in Section \ref{secape}, we are ready to prove Theorems \ref{main1} and \ref{main2}.

\begin{proof}[Proof of Theorem {\rm\ref{main1}}]
Let $(\rho_{0},u_{0},f_{0})$ be the initial data satisfying \eqref{indacd} and \eqref{inenrgcd}. According to Proposition \ref{prop2}, there exists a maximal time $T_{*}>0$ such that the Cauchy problem \eqref{NSV}--\eqref{inda} has a unique strong solution $(\rho,u,f)$
fulfilling \eqref{result1} on $(0,T_{*}]\times\mathbb{R}^{3}\times\mathbb{R}^{3}$. Now we define
\begin{equation}\label{T0}
T_{0}:=\sup\left\{T\in (0,T_{*}]\Big|\int_{0}^T\left(\mu^{-3}\|\nabla u\|_{L^{2}}^{4}+\kappa\|u\|_{L^{\infty}}+20\kappa\|\nabla u\|_{L^{\infty}}\right){\rm{d}}t\leq 1\right\}.
\end{equation}
By \eqref{result1}, there exists a small time $T_{1}\in(0,T_{*})$ such that $T_{0}\geq T_{1}>0$.

We claim that
\begin{equation*}
T_{*}=T_{0}=\infty,
\end{equation*}
as long as \eqref{inenrgcd} holds. Indeed, if $T_{0}<T_{*}$, then one deduces from Proposition \ref{prop3} that
\begin{equation}\label{apriorieee}
\int_{0}^{T_{0}}\left(\mu^{-3}\|\nabla u\|_{L^{2}}^{4}+\kappa\|u\|_{L^{\infty}}+20\kappa\|\nabla u\|_{L^{\infty}}\right){\rm{d}}t\leq \frac{1}{2},
\end{equation}
for all $T\in(0,T_{0}]$. This contradicts the definition of $T_{0}$ in \eqref{T0}, and hence we have $T_{0}=T_{*}$. In order to show $T_{*}=\infty$, we assume $T_{*}<\infty$ toward a contradiction. Since the uniform estimate \eqref{apriorieee} holds for all $T\in(0,T_{*}]$, by virtue of Proposition \ref{prop3} and Lemmas \ref{lem1}, \ref{lem4}, \ref{lem8} and \ref{lem9}, one can take $(\rho,u,f)(t)$ for a time $t$ sufficiently close to $T_{*}$ as new initial data and obtain the existence of solutions beyond the time $T_{*}$ due to Proposition \ref{prop2}, which contradicts the maximality of $T_{*}$. Therefore, $(\rho,u,f)$ is indeed a global strong solution to the Cauchy problem  \eqref{NSV}--\eqref{inda}. In addition, according to \eqref{T0}, Proposition \ref{prop3} and  Lemmas \ref{lem1}--\ref{lem9}, one can show that the global strong solution $(\rho,u,f)$ satisfies the property \eqref{result1}.

To complete the proof of Theorem \ref{main1}, we finally prove the uniqueness issue. For given time $T>0$, let $(\rho_{1},u_{1},P_{1},f_{1})$ and $(\rho_{2},u_{2},P_{2},f_{2})$ be two solutions satisfying \eqref{result1} on $[0,T]$ subject to the same initial data $(\rho_{0},u_{0},f_{0})$. We set
\begin{equation*}
(\tilde{\rho},\tilde{u},\tilde{P},\tilde{f}) =(\rho_{1}-\rho_{2},u_{1}-u_{2}, P_{1}-P_{2},f_{1}-f_{2}),
\end{equation*}
which satisfies
\begin{equation}\label{diso}
\left\{
\begin{aligned}
&\tilde{\rho}_{t}+u_{1}\cdot\nabla_{x}\tilde{\rho}=-\tilde{u}\cdot\nabla_{x}\rho_{2},\\
&\rho_{1} (\tilde{u}_{t}+u_{1}\cdot\nabla_{x}\tilde{u})+\nabla_{x}\tilde{P}-\mu\Delta_{x}\tilde{u}+\int_{\mathbb{R}^{3}}\kappa\tilde{u} f_{1} {\rm{d}}v\\
&=-\tilde{\rho}(u_{2})_{t}-\rho_{1}\tilde{u}\cdot\nabla_{x} u_{2}-\int_{\mathbb{R}^{3}}\kappa(u_{2}-v)\tilde{f} {\rm{d}}v,\\
&{\rm{div}}_{x}\tilde{u}=0,\\
&\tilde{f}_{t}+v\cdot\nabla_{x}\tilde{f}+\kappa {\rm{div}}_{v}( (u_{1}-v)\tilde{f} )=-\kappa\tilde{u} \cdot\nabla_{v}f_{2}.
\end{aligned}
\right.
\end{equation}
Multiplying $\eqref{diso}_{1}$ by $|\tilde{\rho}|^{\frac{1}{2}}{\rm{sign}}\,\tilde{\rho}$ and integrating over $\mathbb{R}^{3}$, we have
\begin{equation*}
\begin{aligned}
\frac{2}{3}\frac{{\rm{d}}}{{\rm{d}}t}\|\tilde{\rho}\|_{L^\frac{3}{2}}^\frac{3}{2}
&\leq \|\nabla u_{1}\|_{L^{\infty}}\|\tilde{\rho}\|_{L^\frac{3}{2}}^\frac{3}{2}+\|\nabla\rho_{2}\|_{L^{2}}\|\tilde{u}\|_{L^{6}}\|\tilde{\rho}\|_{L^{\frac{3}{2}}}^\frac{1}{2}\\
&\leq \|\nabla u_{1}\|_{L^{\infty}}\|\tilde{\rho}\|_{L^\frac{3}{2}}^\frac{3}{2}+C\|\nabla \tilde{u}\|_{L^{2}}\|\tilde{\rho}\|_{L^{\frac{3}{2}}}^\frac{1}{2},
\end{aligned}
\end{equation*}
which leads to
\begin{equation}\label{deltarho}
\begin{aligned}
\|\tilde{\rho}\|_{L^{\frac{3}{2}}}&\leq C\exp{\left(\int_{0}^t\|\nabla u_{1}\|_{L^{\infty}}{\rm{d}}\tau\right)} \int_{0}^t\|\nabla\tilde{u}\|_{L^{2}}{\rm{d}}\tau\\
&\leq Ct^{\frac{1}{2}}\left(\int_{0}^t\|\nabla\tilde{u}\|_{L^{2}}^{2}{\rm{d}}\tau\right)^{\frac{1}{2}}.
\end{aligned}
\end{equation}
Next, multiplying $\eqref{diso}_{4}$ by $|\tilde{f}|^{\frac{1}{5}}{\rm{sign}}\, \tilde{f}$ and then integrating over $\mathbb{R}^{3}\times\mathbb{R}^{3}$, we get
\begin{equation}\label{fffff}
\begin{aligned}
\frac{5}{6}\frac{{\rm{d}}}{{\rm{d}}t}\|\tilde{f}\|_{L^{\frac{6}{5}}_{x,v}}^{\frac{6}{5}}
&\leq\frac{\kappa}{2}\|\tilde{f}\|_{L^{\frac{6}{5}}_{x,v}}^{\frac{6}{5}}+\kappa\|\tilde{u}\|_{L^{6}}\|\nabla_{v}f_{2}\|_{L^{\frac{3}{2}}_{x}(L^{\frac{6}{5}}_{v})}\|\tilde{f}\|_{L^{\frac{6}{5}}_{x,v}}^{\frac{1}{5}}\\
&\leq C\kappa\|\tilde{f}\|_{L^{\frac{6}{5}}_{x,v}}^{\frac{6}{5}}+C\kappa\|\nabla \tilde{u}\|_{L^{2}}\|\tilde{f}\|_{L^{\frac{6}{5}}_{x,v}}^{\frac{1}{5}},
\end{aligned}
\end{equation}
where we have used $\|\nabla_{v}f_{2}\|_{L^{\frac{3}{2}}_{x}(L^{\frac{6}{5}}_{v})}\leq C\|\nabla_{v}f_{2}\|_{L^{\frac{3}{2}}_{x,v}}$ due to the compact support of $f_{2}$ with respect to $v$. Thus, the Gr\"onwall inequality to \eqref{fffff} gives
\begin{equation}\label{deltaf}
\|\tilde{f}\|_{L^{\frac{6}{5}}_{x,v}}\leq C_{T}\int_{0}^{t}\kappa\|\nabla \tilde{u}\|_{L^{2}} {\rm{d}}\tau \leq C_{T}t^{\frac{1}{2}} \left(\int_{0}^{t}\|\nabla \tilde{u}\|_{L^{2}}^{2} {\rm{d}}\tau\right)^{\frac{1}{2}}.
\end{equation}
Since $f$ is compactly supported in $v$, the estimate \eqref{deltaf} implies
\begin{equation}\label{deltanj}
\|(n_{\tilde{f}},j_{\tilde{f}})\|_{L^{\frac{6}{5}}}\leq C_{T}t^{\frac{1}{2}} \left(\int_{0}^{t}\|\nabla \tilde{u}\|_{L^{2}}^{2} {\rm{d}}\tau\right)^{\frac{1}{2}}.
\end{equation}
Furthermore, multiplying $\eqref{diso}_{2}$ by $\tilde{u}$, integrating the resulting equation by parts and using \eqref{deltarho} and \eqref{deltanj}, we deduce that
\begin{equation*}
\begin{aligned}
&\frac{1}{2}\frac{{\rm{d}}}{{\rm{d}}t}\int_{\mathbb{R}^{3}}\rho_{1}|\tilde{u}|^{2}{\rm{d}}x+\int_{\mathbb{R}^{3}}\mu|\nabla\tilde{u}|^{2}{\rm{d}}x+\int_{\mathbb{R}^{3}}\kappa n_{f_{1}}|\tilde{u}|^{2}{\rm{d}}x\\
&\leq\|(u_{2})_{t}\|_{L^{6}}\|\tilde{\rho}\|_{L^\frac{3}{2}}\|\tilde{u}\|_{L^{6}}+\|\nabla u_{2}\|_{L^{\infty}}\int_{\mathbb{R}^{3}}\rho_{1}|\tilde{u}|^{2}{\rm{d}}x\\
&\quad+\|u_{2}\|_{L^{6}}\|\nabla u_{2}\|_{L^{\infty}}\|\tilde{\rho}\|_{L^\frac{3}{2}}\|\tilde{u}\|_{L^{6}}+\|u_{2}\|_{L^{\infty}}\|\tilde{u}\|_{L^{6}}\|n_{\tilde{f}}\|_{L^\frac{6}{5}}+\|\tilde{u}\|_{L^{6}}\|j_{\tilde{f}}\|_{L^\frac{6}{5}}\\
&\leq C_{T}\Big(\|\nabla u\|_{L^{\infty}}+t\|\nabla (u_{2})_{t}\|_{L^{2}}^{2}+t\|\nabla u_{2}\|_{H^{1}\cap W^{1,6}}^{2}\Big)\int_{0}^{t}\|\nabla \tilde{u}\|_{L^{2}}^{2} {\rm{d}}\tau+\frac{\mu}{2}\int_{\mathbb{R}^{3}}|\nabla\tilde{u}|^{2}{\rm{d}}x,
\end{aligned}
\end{equation*}
from which we obtain for $t\in(0,T]$ that
\begin{equation}\label{difl}
\frac{{\rm{d}}}{{\rm{d}}t}\widetilde{G}(t)\leq C_{T}\Big(\|\nabla u\|_{L^{\infty}}+t\|\nabla (u_{2})_{t}\|_{L^{2}}^{2}+t\|\nabla u_{2}\|_{H^{1}\cap W^{1,6}}^{2} \Big)\widetilde{G}(t),
\end{equation}
where $\widetilde{G}(t)$ is defined by
\begin{equation*}
\widetilde{G}(t):=\int_{\mathbb{R}^{3}}\rho_{1}|\tilde{u}|^{2}{\rm{d}}x+\int_{0}^{t}\int_{\mathbb{R}^{3}}\mu|\nabla\tilde{u}|^{2}{\rm{d}}x+\int_{\mathbb{R}^{3}}\kappa n_{f_{1}}|\tilde{u}|^{2}{\rm{d}}x.
\end{equation*}
Due to $t^{\frac{1}{2}} \nabla (u_{2})_{t}\in L^{2}(0,T;L^{2})$ and $t^{\frac{1}{2}} \nabla u_{2}\in L^{2}(0,T;H^{1}\cap W^{1,6})$, making use of the Gr\"onwall inequality on \eqref{difl}, one has $\widetilde{G}(t)\equiv 0$ for $t\in (0,T]$. This together with \eqref{deltarho} and \eqref{deltaf}  implies that $(\rho_{1},u_{1},f_{1})\equiv(\rho_{2},u_{2},f_{2})$. Finally, it follows from the Stokes estimate on $\eqref{diso}_{2}$ that $P_{1}\equiv P_{2}$, and thus the proof of Theorem \ref{main1} is completed.
\end{proof}

\begin{proof}[Proof of Theorem {\rm\ref{main2}}]
Let $(\rho_{0},u_{0},f_{0})$ be the initial data satisfying \eqref{indacd2} and \eqref{inenrgcd2}. For any $\varepsilon\in(0,1)$, we regularize the initial data as follows:
\begin{equation}\label{weakdata}
(\rho_{0}^{\varepsilon}(x),u_{0}^{\varepsilon}(x),f_{0}^{\varepsilon}(x,v)):=(J_{1}^{\varepsilon}\ast\rho_{0}(x)+\varepsilon,J_{1}^{\varepsilon}\ast u_{0}(x),J_{1}^{\varepsilon}\ast J_{2}^{\varepsilon}\ast(f_{0}\mathbf{1}_{|v|\leq \varepsilon^{-1}})(x,v)),
\end{equation}
where $J_{1}^{\varepsilon}(x)$ and $J_{2}^{\varepsilon}(v)$ are the Friedrichs mollifier with respect to the variables $x$ and $v$, respectively, and $\mathbf{1}_{|v|\leq \varepsilon^{-1}}\in C_{0}^{\infty}(\mathbb{R})$ is the cut-off function such that $\mathbf{1}_{|v|\leq \varepsilon^{-1}}=1$ for $|v|\leq \varepsilon^{-1}$, and $\mathbf{1}_{|v|\leq \varepsilon^{-1}}=0$ for $|v|\geq 2\varepsilon^{-1}$. It is easy to verify that $(\rho_{0}^{\varepsilon}, u_{0}^{\varepsilon},f_{0}^{\varepsilon})$ satisfies the assumptions \eqref{indacd2}--\eqref{inenrgcd2} uniformly in $\varepsilon\in(0,1)$ and as $\varepsilon\rightarrow 0$, one has
\begin{equation*}
\left\{
\begin{aligned}
&\begin{aligned}
&\rho^{\varepsilon}_{0}\rightarrow \rho_{0}&&\text{in}~~L^{p} , &&\frac{3}{2}\leq p<\infty,\\
&\nabla u_{0}^{\varepsilon}\rightarrow \nabla u_{0}&&\text{in}~~ L^{2},\\
&f_{0}^{\varepsilon}\rightarrow  f_{0}&&\text{in}~~ L^{p}_{x,v} , &&1\leq p<\infty,
\end{aligned}\\
&\frac{1}{2}\int_{\mathbb{R}^{3}}\rho^{\varepsilon}_{0}|u^{\varepsilon}_{0}|^{2}{\rm{d}}x+\frac{1}{2}\int_{\mathbb{R}^{3}\times\mathbb{R}^{3}}|v|^{2}f^{\varepsilon}_{0} {\rm{d}}v{\rm{d}}x\rightarrow E_{0}.
\end{aligned}
\right.
\end{equation*}
In addition, it holds for $q>5$ that
\begin{equation*}
\begin{aligned}
\|(1+|v|^{2})f^{\varepsilon}_{0}\|_{L^{1}_{v}(L^{\infty}_{x})}&\leq \| (1+|v|^{q})f^{\varepsilon}_{0}\|_{L^{\infty}_{x,v}}\int_{\mathbb{R}^{3}}\frac{1+|v|^{2}}{1+|v|^{q}}{\rm{d}}v\\
&\leq C\|(1+|v|^{q})f_{0}\|_{L^{\infty}_{x,v}}.
\end{aligned}
\end{equation*}

Owing to Theorem \ref{main1}, there exists a unique global strong solution $(\rho^{\varepsilon},u^{\varepsilon}, f^{\varepsilon})$ for any $\varepsilon\in (0,1)$ to the system \eqref{NSV} with the initial data $(\rho_{0}^{\varepsilon},u_{0}^{\varepsilon},f_{0}^{\varepsilon})$. For any given time $T>0$, it follows from Proposition \ref{prop3} and Lemmas \ref{lem1}, \ref{lem3}, \ref{lem4} that $(\rho^{\varepsilon},u^{\varepsilon},P^{\varepsilon},f^{\varepsilon})$ has the uniform bounds \eqref{clspri}, \eqref{rhoinfty}, \eqref{eptd}, \eqref{dsin}, \eqref{exdsin}, \eqref{twds}, \eqref{extwds} and \eqref{PL2}--\eqref{extwdsjj}. Thus, by virtue of the Aubin-Lions lemma and the Cantor diagonal argument, there exists a limit $(\rho,u,P,f)$ such that up to a subsequence, it holds as $\varepsilon\rightarrow0$ that
\begin{equation*}
\left\{
\begin{aligned}
&\rho^{\varepsilon}\overset{\ast}{\rightharpoonup}\rho&&\text{in}~~ L^{\infty}(0,T;L^{\frac{3}{2}}\cap L^{\infty}),\\
&u^{\varepsilon}\rightarrow u&&\text{in}~~ L^{2}(0,T;H^{1}_{loc}),\\
&P^{\varepsilon}\overset{}{\rightharpoonup} P&&\text{in}~~ L^{2}(0,T;H^{1}),\\
&f^{\varepsilon}\overset{\ast}{\rightharpoonup} f&&\text{in}~~ L^{\infty}(0,T;L^{\infty}_{x,v}),
\end{aligned}
\right.
\end{equation*}
and
\begin{equation*}
\int_{\mathbb{R}^{3}}(u^{\varepsilon}-v)f^{\varepsilon}{\rm{d}}v\overset{}{\rightharpoonup} \int_{\mathbb{R}^{3}}(u-v)f{\rm{d}}v\quad\text{in}~~\mathcal{D}^{\prime}((0,T)\times\mathbb{R}^{3}).
\end{equation*}
Therefore, it is easy to check that $(\rho,u,f)$ solves the system \eqref{NSV} in the sense of distributions. Due to the Fatou property, $(\rho,u,f)$ satisfies the uniform estimates in Lemmas \ref{lem1}--\ref{lem5} and \ref{lem7}, the large-time behavior $\eqref{r2}_{1}$ and the energy inequality \eqref{enrgineq} for a.e. $t\in(0,T]$. From the renormalized arguments for $\eqref{NSV}_{1}$ and $\eqref{NSV}_{4}$ (refer to \cite{Lions,Bouchut}), one has $\rho\in C([0,T];L^{\frac{3}{2}})$ and $f\in C([0,T];L^{1}_{x,v})$. It follows from $\rho u\in L^{\infty}(0,T;L^{2})$ and $(\rho u)_{t}\in L^{\infty}(0,T;H^{-1})$ that $\rho u \in C([0,T];H^{-1})$, and therefore $(\rho,u,f)$ is indeed a global weak solution to the Cauchy problem \eqref{NSV}--\eqref{inda} in the sense of Definition \ref{defwkslt}.  Repeating the same argument as in Lemma \ref{lem6}, we have the asymptotic behaviors $\eqref{r2}_{2}$--$\eqref{r2}_{3}$. Thus, the proof of Theorem \ref{main2} is completed.
\end{proof}

\section*{Appendix}\label{appendix}
\setcounter{equation}{0}
\renewcommand{\theequation}{A.\arabic{equation}}
\setcounter{section}{0}
\renewcommand{\thesection}{A}

In this section, we provide some auxiliary lemmas that are used frequently throughout this paper. We first recall the Gagliardo-Nirenberg inequality (e.g., refer to \cite{Nirenberg}).

\begin{lemma}\label{GNineq}
Let $f\in L^{q}(\mathbb{R}^{3})$ and $D^{s}f\in L^r(\mathbb{R}^{3})$ with $1\leq q,\,r\leq\infty$. For $0\leq k<s$, there exists a constant $C>0$ such that
\begin{equation}\label{GNSine}
\|D^kf\|_{L^p}\leq C\|D^sf\|_{L^r}^\theta\|f\|_{L^{q}}^{1-\theta},
\end{equation}
where
\begin{equation*}
\frac{1}{p}-\frac{k}{3}=\left(\frac{1}{r}-\frac{s}{3}\right)\theta+\frac{1}{q}(1-\theta)~~\text{and}~~ \frac{k}{s}\leq\theta<1.
\end{equation*}
\end{lemma}

\vspace{2mm}

Next, we consider the Cauchy problem for the Stokes equations below
\begin{equation}\label{steq}
\left\{
\begin{aligned}
&-\mu\Delta u+\nabla P=g,\quad x\in\mathbb{R}^{3},\\
&{\rm{div}}\, u=0,\\
&u(x)\rightarrow 0~~\text{as}~~|x|\rightarrow{\infty},
\end{aligned}
\right.
\end{equation}
and we have the following regularity estimates.

\begin{lemma}\label{SSineq}
If $g\in L^{\frac{6}{5}}\cap L^p$ with $p\in(\frac{6}{5},\infty)$, there exists some positive constant $C$ depending only on $p$ such that a unique weak solution $(u,P)\in D_{0,\sigma}^{1}\times L^{2}$ to the Cauchy problem \eqref{steq} exists and satisfies
\begin{equation}\label{re1}
\mu\|\nabla u\|_{L^{2}}+\|P\|_{L^{2}}\leq C\|g\|_{L^{\frac{6}{5}}},
\end{equation}
and
\begin{equation}\label{re2}
\mu\|\nabla^{2}u\|_{L^p}+\|\nabla P\|_{L^p}\leq C\|g\|_{L^p}.
\end{equation}
\end{lemma}
\begin{proof}
The existence and uniqueness of the solution to \eqref{steq} is classical, e.g., cf. \cite{Galdi}. Multiplying $\eqref{steq}_{1}$ by $u$ and then integrating by parts together with $\eqref{steq}_{2}$, we obtain
\begin{equation*}
\mu\|\nabla u\|_{L^{2}}^{2}=\int_{\mathbb{R}^{3}} g\cdot u{\rm{d}}x\leq\|g\|_{L^{\frac{6}{5}}}\|u\|_{L^{6}}\leq C\|g\|_{L^{\frac{6}{5}}}\|\nabla u\|_{L^{2}},
\end{equation*}
which implies
\begin{equation*}
\mu\|\nabla u\|_{L^{2}}\leq C\|g\|_{L^{\frac{6}{5}}}.
\end{equation*}
In addition, we can derive that
\begin{equation*}
P=-\mu(-\Delta)^{-1}{\rm{div}}\Delta u-(-\Delta)^{-1}{\rm{div}}\,g,
\end{equation*}
which combined with the Sobolev inequality leads to
\begin{equation*}
\|P\|_{L^{2}}\leq\|(-\Delta)^{-1}{\rm{div}}\, g\|_{L^{2}}\leq C\|g\|_{L^{\frac{6}{5}}}.
\end{equation*}
Thus, one can conclude \eqref{re1}. A similar argument leads to \eqref{re2}.
\end{proof}

We also need the following version of the inverse function theorem, which can be shown by direct computations (or see \cite{HanKwan2020}).

\begin{lemma}\label{IFthm}
For $\Omega=\mathbb{R}^{3}$, if $\phi:\Omega\rightarrow\Omega$ is $C^{1}$ and satisfies $\|\nabla\phi\|_{L^{\infty}}<1$, then $f:={\rm{Id}}+\phi$ is a $C^{1}$-diffeomorphism of $\Omega$ onto itself satisfying $\|\nabla f\|_{L^{\infty}}\leq\left(1-\|\nabla\phi\|_{L^{\infty}}\right)^{-1}$. If furthermore $\|\nabla\phi\|_{L^{\infty}}\leq\frac{1}{9}$, then ${\rm{det}}\,\nabla f\geq\frac{1}{2}$.
\end{lemma}

Finally, we recall the definition of the Wasserstein distances and the Monge-Kantorovitch duality. The interested reader can refer to the book \cite{Villani}.

\begin{definition}\label{defwass}
Denote either $\mathbb{R}^{3}$ or $\mathbb{R}^{3}\times\mathbb{R}^{3}$ by $X$ and let $\mu_{1},\mu_{2}$ be the Borel measures on $X$. The $p$-Wasserstein distance between $\mu_{1}$ and $\mu_{2}$ is defined by
\begin{equation*}
W_p(\mu_{1},\mu_{2}):=\left(\inf_{\omega\in\Gamma(\mu_{1},\mu_{2})}\int_{X\times X}|z-z'|^pd\omega(z,z')\right)^{\frac{1}{p}},
\end{equation*}
where $\Gamma(\mu_{1},\mu_{2})$ represents the collection of all measures on $X\times X$ that the first and second marginal equal to $\mu_{1}$ and $\mu_{2}$, respectively.
\end{definition}

\begin{lemma}\label{MKdual}
Fix $p=1$ and $\mu_{1},\mu_{2}$ be the Borel measures on $X$, then it holds
\begin{equation*}
\begin{aligned}
&W_{1}(\mu_{1},\mu_{2})\\
&=\sup\left\{\int_{X}\psi(z)d\mu_{1}(z)-\int_{X}\psi(z)d\mu_{2}(z)\Big|~\forall\psi\in{\rm{Lip}}(X),~\|\nabla\psi\|_{L^{\infty}(X)}\leq 1\right\}.
\end{aligned}
\end{equation*}
\end{lemma}

\vspace{5mm}

\noindent\textbf{Acknowledgments.} 
The second author is grateful to Professor Rapha\"el  Danchin for the suggestion on the uniqueness of weak solutions constructed in Theorem \ref{main2} when visiting the LAMA in Universit\'e Paris-Est Cr\'eteil.

\end{document}